\newif\iflongversion\longversionfalse % Marcus
\newif\ifoldversion\oldversionfalse % Marcus

\documentclass[11pt]{article}

\usepackage{amsmath}
\usepackage{amssymb}

\usepackage[shortlabels]{enumitem}%
\setlist[enumerate,1]{label={\rm (\arabic*)}}% Marcus
\setlist[enumerate]{itemsep=0pt}% Marcus
\usepackage[pdftex,usenames,x11names]{xcolor}%
\usepackage[pdftex,bookmarksnumbered=true,backref]{hyperref}
\usepackage[pdftex]{graphicx,color}
\usepackage[a4paper]{geometry}

\title{Defining integer valued functions in rings of continuous definable
functions over a topological field}

\author{Luck Darni\`ere%
  \and
  Marcus Tressl%\footnote{Manchester (to be completed)}
}

\reversemarginpar

\newcounter{numtodo}
\newcommand{\TODO}[3]{%
  {\bigskip\bgroup\parindent=0pt\setlength{\parskip}{\smallskipamount}%
    \colorbox{#3}{\color{white}\bf #1~\thenumtodo}\enskip%
    \color{#3} #2}\stepcounter{numtodo}\egroup\bigskip}

\usepackage{ntheorem}
\theoremnumbering{Alph}
\newtheorem{thm-Alph}{Theorem}
\newtheorem*{prop-nonum}{Proposition}

\newtheorem{lemma}[subsection]{Lemma}
\newtheorem{proposition}[subsection]{Proposition}
\newtheorem{property}[subsection]{Property}
\newtheorem{theorem}[subsection]{Theorem}
\newtheorem{corollary}[subsection]{Corollary}
\newenvironment{remark}%
   {\refstepcounter{subsection}%
        \medbreak\noindent{\bf Remark \thesubsection\space}}%
   {\par\medbreak}%
   {\refstepcounter{subsection}%
        \medbreak\noindent{\bf Observation \thesubsection\space}}%
   {\par\medbreak}%
   {\refstepcounter{subsection}%
        \medbreak\noindent{\bf Example \thesubsection\space}}%
   {\par\medbreak}%
\newenvironment{examples}% Marcus
   {\refstepcounter{subsection}%
        \medbreak\noindent{\bf Examples \thesubsection\space}}%
   {\par\medbreak}%

\newenvironment{definition}%
   {\refstepcounter{subsection}%
        \medbreak\noindent{\bf Definition \thesubsection\space}}%
   {\par\medbreak}%

\newenvironment{proof}%
   {\medbreak\noindent{\it Proof:\space}}%
   {\QED}% Marcus

\newcommand{\df}{\bf}
\renewcommand{\cal}{\mathcal}
\renewcommand{\leq}{\leqslant}
\renewcommand{\geq}{\geqslant}
\let\sauvegardetiret=\-
\renewcommand{\-}[1]{\ifx#1-\penalty10000\hbox{-\relax}\penalty10000\else\sauvegardetiret#1\fi}
\newcommand{\tq}{\colon}
\newcommand{\st}{\ \vert \ }% Marcus
\newcommand{\mal}{\!\cdot\!}% Marcus
\newcommand{\eps}{\varepsilon}% Marcus
\def\phi{\varphi}% Marcus
\newcommand{\Spec}{\Sp}% Marcus
\newcommand\qed{\nolinebreak\hspace*{\fill}\text{$\square$}}% Marcus
\newcommand\QED{\qed}

\newcommand{\Dm}{{\mathfrak{m}}}% Marcus

\newcommand{\NN}{{\mathbb N}}
\newcommand{\ZZ}{{\mathbb Z}}
\newcommand{\QQ}{{\mathbb Q}}
\newcommand{\RR}{{\mathbb R}}
\newcommand{\CC}{{\mathbb C}}
\newcommand{\C}{{\CC}}% Marcus
\newcommand{\cB}{{\cal B}}

\newcommand{\cK}{{\cal K}}
\newcommand{\cL}{{\cal L}}

\newcommand{\cO}{{\cal O}}

\newcommand{\cZ}{{\cal Z}}
\newcommand{\lra}{\longrightarrow}% Marcus
\newcommand{\0}{\emptyset}% Marcus
\newcommand\OldStart{\color{gray}[BEGIN OLD VERSION]}% Marcus
\newcommand\OldEnd{[END OLD VERSION]\color{black}}% Marcus
\newcommand\LongStart{\color{teal}[BEGIN LONG VERSION]}% Marcus
\newcommand\LongEnd{[END LONG VERSION]\color{black}}% Marcus
\newcommand{\Lring}{{\cal L}_{\rm ring}}
\newcommand{\Loag}{{\cal L}_{\rm og}}
\newcommand{\CF}{{\cal C}}
\newcommand{\ZS}{{\cal V}}

\newcommand{\Sp}{\mathop{\rm Spec}}

\newcommand{\Jac}{\mathop{\rm Jac}}

\newcommand{\FORM}[1]{\mbox{\rm #1}}
\newcommand{\aInt}{\FORM{Int}^+_\preccurlyeq}
\newcommand{\mInt}{\FORM{Int}^\times}
\newcommand{\alimch}{\FORM{LimCh}^+_\preccurlyeq}
\newcommand{\mlimch}{\FORM{LimCh}^\times}
\newcommand{\alimBch}{\FORM{LimBCh}^+_\preccurlyeq}
\newcommand{\mlimBch}{\FORM{LimBCh}^\times}
\newcommand{\limit}{\FORM{Limit}}
\newcommand{\inter}{\FORM{Inter}}
\newcommand{\isol}{\FORM{Isol}}
\newcommand{\point}{\FORM{Point}}

\newenvironment{listD}%
  {\begin{list}{\bf(Dim\theenumi)}{\usecounter{enumi}}}%
  {\end{list}}
\begin{document}

\maketitle

\begin{abstract}
\noindent
  Let $\cK$ be an expansion of either an ordered field $(K,\leq)$, or a
  valued field $(K,v)$. Given a definable set $X\subseteq K^m$ let $\CF(X)$ be
  the ring of continuous definable functions from $X$ to $K$.
  %We first
%  prove a ``partitions of unity'' result: if $f,g\in\CF(X)$ have no
%  common zero there are $a,b\in\CF(X)$ such that $af+bg=1$.
  Under very mild assumptions on the geometry of $X$
  and on the structure $\cK$, in particular when $\cK$ is
  $o$\--minimal or $P$\--minimal, or an expansion of a local field,
  we prove that the ring of integers
  $\ZZ$ is interpretable in $\CF(X)$.
  If $\cK$ is $o$-minimal and $X$ is definably connected
  of pure dimension $\geq2$, then $\CF(X)$ defines the subring $\ZZ$.
  If $\cK$ is $P$-minimal and $X$ has no isolated points, then
    there is a discrete ring
  ${\cal Z}$ contained in $K$ and naturally isomorphic to $\ZZ$, such that the ring of functions
  $f\in\CF(X)$ which take values in ${\cal Z}$ is definable in $\CF(X)$.
\end{abstract}

\tableofcontents

\section{Introduction}

We give a first order definition of the ring of integers
within rings of continuous functions that are first order definable in expansions of ordered and valued fields.
Before describing a more technical outline of the contents, we explain the context of the article and the results
in a colloquial way.

Rings of continuous functions on topological spaces are central objects in functional analysis, topology and geometry. To name an example: They are  rings of sections for the sheaf of continuous (say, real valued)  functions on a topological space and as such play the algebraic role in the study of topological (Hausdorff) spaces.

By a  ring of \textit{definable} continuous functions we mean the following. Let $\cK$ be an expansion  of an ordered field $(K,\leq)$ (e.g. the real field) or a valued field $(K,v)$ (e.g. the $p$\--adics). In both cases $K$ carries a topology that turns $K$ into a topological field.
Let $X\subseteq K^n$ be definable (with parameters) in $\cK$ and let $C(X)$ be the set of all functions $X\lra K$ that are continuous and definable in $\cK$. If $\cK$ is the real or the $p$\--adic field, definable is the same as semi-algebraic.
Then $C(X)$ is a ring and similar to the classical case mentioned above,
$C(X)$ carries the algebraic information of the \textit{definable} homeomorphism type of $X$. This is amply illustrated in the case when $\cK$ is a real closed field:
Let $\mathtt{S}$ be the category of semi-algebraic subsets $X\subseteq \cK^n$, $n\in\NN$, with
continuous semi-algebraic maps as morphisms. Let $\mathtt{C}$ be the category of all the rings $C(X)$ of continuous $\cK$-definable functions and $\cK$-algebra homomorphisms as morphisms. Then the functor $C:\mathtt{S}\lra \mathtt{C}$ that sends $X$ to $C(X)$
and a morphism $f:X\lra Y$ to the $\cK$-algebra homomorphism $C(Y)\lra C(X);\ g\mapsto g\circ f$, is an anti-equivalence
of categories. This can also be seen in analogy with algebraic geometry, where varieties defined over a field $K$ are anti-equivalent to affine $K$-algebras, leading to the modern language of schemes. For semi-algebraic sets the machinery is developed in this vain in terms of so called semi-algebraic spaces, see \cite{schw-1987}.
The same connection exists between definable sets $X$ in the $p$\--adic context and the rings $C(X)$ for that context.
In each case, one may thus think of rings of continuous definable functions as
``coordinate rings" when studying topological properties of definable sets.

Model theoretic studies of rings of continuous (definable) functions
may be found in \cite{pre-schw-2002}, \cite{tressl-2007} for the case of real closed fields and in
\cite{belair-1991},\cite{belair-1995}, \cite{guzy-tressl-2008} in the $p$\--adic case. Model theory of rings of continuous functions on topological spaces have been studied in \cite{cherlin-1980}, which serves as a main source for inspiration for us.

In this article we study model theoretic properties of rings of continuous definable functions under fairly mild assumption on the base structure.
The principal goal is to show that,
in most cases, these rings \textit{interpret} the ring of integers in a uniform way; in particular these rings are not decidable.
This is established in Theorem~\ref{th:def-chunk}.
Now undecidability  was already known in some cases. For example, one can interpret the lattice of closed subsets of $\RR^n$ in the ring $C(\RR^n)$ of continuous semi-algebraic functions $\RR^n\lra \RR$.
When $n\geq 2$, this lattice itself is undecidable by \cite{grzegorczyk-1951} and indeed interprets the ring of integers, see \cite{tressl-2017}.
However, in this lattice, one cannot interpret $C(\RR^n)$ in a uniform
way as we will see in Remark~\ref{re:CofXNotInLattice}.

In the $p$\--adic case, the first author has shown in \cite{darniere-2006}
that the lattice of closed subsets of $\QQ_p^n$ is even decidable and it was unknown whether $C(\QQ_p^n)$ has a decidable theory at all.

In section \ref{se:loc-dim} we show that
in many cases the rings $C(X)$ actually \textit{define} the subset of
constant functions with values in a natural isomorphic copy of $\ZZ$,
see Theorem~\ref{th:def-fns-Dk-Z} for the precise formulation.
When $\cK$ is an $o$\--minimal expansion of a field and $X$ is definably connected of local dimension $\geq 2$ everywhere, then indeed $C(X)$ defines the ring of constant functions with values in $\ZZ$.
As a consequence, when $\cK$ expands the real field
we obtain that the real field, seen as constant functions $\RR^n\lra \RR$
is definable in $C(X)$. This implies that the projective hierarchy is definable in
these rings.

\medskip\noindent
\textbf{Detailed description of the main results and set up.}

\medskip\noindent
We consider an expansion $\cK$ of a topological field $(K,\cO)$ where
$\cO$ is either the unit interval $[-1,1]$ of a total order $\leq$ on
$K$, or the ring of a non-trivial valuation $v$ on $K$. We endow $K^m$
with the corresponding topology, for every integer $m\geq0$. We will make
almost everywhere the following assumption on $\cK$.
\begin{description}
\hypertarget{BFin}{}
  \item[(BFin)]
    \quad Every definable subset of $K$ that is closed, bounded and discrete, is finite.
\end{description}
As is well known, every discrete definable subset of $K$ is finite if
$\cK$ is $o$\--minimal \cite{drie-1998}, $P$\--minimal
\cite{hask-macp-1997} or $C$\--minimal \cite{hask-macp-1994}, and more
generally if it is a dp-minimal ordered or valued field \cite{simo-2011},
\cite{jahn-simo-wals-2017}. The same holds true if $\cK$ is any
Henselian (non-trivially) valued field of characteristic $0$, or any
algebraically bounded expansion of such \cite{drie-1989}. But property
\hyperlink{BFin}{(BFin)} is also obviously satisfied if $\cO$ is compact, hence if
$(K,\cO)$ is {\em any} expansion of $\RR$, the field of real numbers, or
any valued local field (that is a finite extension of the field $\QQ_p$
of $p$\--adic numbers, or the field of Laurent series $F((t))$ over
a finite field $F$).

Given any two definable sets $X\subseteq K^m$ and $Y\subseteq K^n$ we let $\CF(X,Y)$
denote the set of continuous functions  $X\to Y$ that are definable with parameters
in $\cK$. If $Y$ is a subring of $K$, e.g. when $Y=\ZZ$,
then  $\CF(X,Y)$ is considered as a ring, where addition and multiplication is given point-wise.
When $Y=K$, we just write $\CF(X)$.
Let $\tau^\ZZ=\big\{\tau^k\tq k\in\ZZ\big\}$ for
some non-zero $\tau\in\cO$ with $1/\tau\notin\cO$.
We also furnish $\tau^\ZZ$ with a ring structure so that the bijection
$\ZZ\lra \tau^\ZZ,\ k\mapsto \tau ^k$ is an isomorphism.
Again, pointwise addition and multiplication turns $\CF(X,\tau^\ZZ)$ into a ring.

Our main result is
that, under very general conditions on $\cK$ and $X$, the ring
$\CF(X,\ZZ)$ (resp. $\CF(X,\tau^\ZZ)$) is definable (resp. interpretable) in
the ring structure of $\CF(X)$ expanded by the set
\begin{displaymath}
  \cB =\{s\in \CF(X)\st \forall x\in X: s(x)\in\cO\}.
\end{displaymath}
\noindent
In many cases, in particular when $\cK$ is $o$\--minimal case or an expansion of a $p$\--adically closed field, we will see that $\cB$ is already definable in $\CF(X)$, see Remark~\ref{re:def-O3}.
On the other hand $\cO$ will {\em not} be definable in the ring structure when $K$ is algebraically closed.

\smallskip
A crucial input and starting point of the paper is the definability of the Nullstellensatz in a weak,
but surprisingly general form in Theorem~\ref{th:PosetZeroIsJac}. This says that
for almost \textit{any} ring $A$ of functions from some given set $S$ to a given field $K$,
the $n+1$-ary relation
$\{f_1=0\}\cap \ldots\cap \{f_n=0\}\subseteq \{g=0\}$ of $A$
is the Jacobson radical relation of $A$. For example it suffices to ask that
$K$ is not algebraically closed,
$A$ contains the constant functions with value in $K$ and that all functions in $A$ without zero
are invertible. This is explained in section \ref{se:SectionDefZeroPoset}.
In the case of $o$\--minimal or $P$\--minimal structures it implies that the lattice of closed definable subsets
of $X$ is interpretable in $\CF(X)$.

\medskip
The technical heart of this paper is Section~\ref{se:lim-chunk} where
we prove our first main result (Theorem~\ref{th:def-chunk} and
Corollary~\ref{co:def-Z-p0}).

\begin{thm-Alph}
  Assume that $\cK$ satisfies \hyperlink{BFin}{(BFin)}. Let $\tau$ be a non-zero and
  non-invertible element of $\cO$. Let $X\subseteq K^m$ be a definable set
  which has arbitrarily many germs\footnote{Roughly speaking, $X$ has
    arbitrarily many germs at $p_0$ if there exists arbitrarily many
    disjoint definable subsets of $S$ of $X\setminus\{p_0\}$ such that
    $p_0\in\overline{S}$. See Section~\ref{se:notation} for a precise
  definition.}
  at some point $p_0$.
  \begin{itemize}[itemsep=0pt]
    \item
      If $\cO=[-1,1]$ then the ring of functions $f\in\CF(X)$ such that
      $f(p_0)\in\ZZ$ is definable in $(\CF(X),\cB)$.
    \item
      In any case the set of functions $f\in\CF(X)$ such that
      $f(p_0)\in\tau^\ZZ$ is definable in $(\CF(X),\cB)$, and its natural
      ring structure is interpretable in $(\CF(X),\cB)$.
  \end{itemize}
  As a consequence the ring of integers $\ZZ$ is interpretable in
  $(\CF(X),\cB)$.
\end{thm-Alph}

If $\cK'=(K',\dots)$ is an elementary extension of $\cK$, and $X'$ is the
subset of $K'^m$ defined by the same formula as $X$, there is a
natural embedding of $(\CF(X),\cB)$ into $(\CF(X'),\cB')$. It follows
from the above result that, surprisingly enough, this is not an
elementary embedding in general (Corollary~\ref{co:plong-non-elem}).

Note that the above theorem is fairly general: it only assumes that $\cK$
satisfies \hyperlink{BFin}{(BFin)} and $X$ has arbitrarily many germs at $p_0$.
In Section~\ref{se:loc-dim} we improve it by assuming, in addition to \hyperlink{BFin}{(BFin)}, that $\cK$
has a good dimension theory for definable sets (see the axioms list
\hyperlink{Dim}{(Dim)} in Section~\ref{se:notation}). This holds true if $\cK$ is
$o$\--minimal, $P$\--minimal or $C$\--minimal, and more generally if
it is dp-minimal (see Remark~\ref{re:dim-dp}). We can then prove our
second main result.

\begin{thm-Alph}
  Assume that $\cK$ satisfies \hyperlink{Dim}{(Dim)} and \hyperlink{BFin}{(BFin)}. Let $\tau$ be a non-zero
  and non-invertible element of $\cO$. Let $X\subseteq K^m$ be a definable
  set of pure dimension $d\geq2$.
  \begin{enumerate}
    \item
      If $\cO=[-1,1]$ then $\CF(X,\ZZ)$ is definable in
      $(\CF(X),\cB)$.
    \item
      In any case $\CF(X,\tau^\ZZ)$ is definable in $(\CF(X),\cB)$.
  \end{enumerate}
\end{thm-Alph}

Theorem~\ref{th:def-fns-Dk-Z} actually gives a more precise and more
general statement. Let us also mention that in the $P$\--minimal case
(among others, see condition \hyperlink{ConditionZ}{(Z)} in Section~\ref{se:loc-dim}) the
condition on the pure dimension of $X$ can be relaxed: the result
holds true whenever $X$ has no isolated point.

\section{Model theoretic and topological set up}
\label{se:notation}

Let $K$ be a field and let $\cO$ be either the unit interval
$[-1,1]$ of a total order $\leq$ of $K$, or the ring of a non-trivial
valuation $v$ of $K$.

In both cases we let $\cO^\times$ denote the set of non-zero $a\in \cO$ with $a^{-1}\in\cO$.
This is a multiplicative subgroup of $K^\times=K\setminus\{0\}$. We
let $|\ |:K^\times\to K^\times/\cO^\times$ be the residue map and extend it
by $|0|=0$. In the ordered case this is just the usual
absolute value, in the valued case $|x|$ is a multiplicative notation
for the valuation. The set $|K|=\{|x|\tq x\in K\}$ is totally ordered by
the relation $|y|\leq |x|$ if and only if $y\in x\mal \cO$. Note that in the
valued case, $|y|\leq|x|$ if and only if $v(y)\geq v(x)$. The
multiplication defined on $|K|$ by $|x|\cdot|y|=|xy|$ extends the
multiplication of $|K|\setminus\{0\}=K^\times/\cO^\times$. The latter is a totally
ordered abelian (multiplicative) group. We denote it by $|K^\times|$, or
$v(K^\times)$ when additive notation is more appropriate.

For every $x\in K^m$ we set $\|x\|=\max(|x_1|,\dots,|x_m|)$, and
for all $X\subseteq K^m$ we write $\|X\|=\{\|x\|\tq x\in X\}$; if $m=1$ we simply
write $|X|$ (or $v(X)$ in additive notation). Open and closed {\df
balls} in $K^m$ with center $c\in K^m$ and radius $r\in K^\times$ are defined
as usually; both are clopen in the valued case. We endow $K^m$
with the topology defined by the open balls, and $|K|$ with the image
of this topology (which induces the discrete topology on $|K^\times|$ in
the valued case). For any set $S$, we write $\overline{S}$ for the
{\df topological closure} of $S$ and $\partial S=\overline{S}\setminus S$ for the {\df frontier} of $S$.

In this paper, {\df definable} means ``first-order definable
with parameters''. Let $\Loag=\{e,*,\leq\}$ be the language of (additive or
multiplicative) ordered groups, let $\Lring=\{0,1,+,-,\times\}$ be the language
of rings and let $\cL$ be an extension of $\Lring$ containing a unary
predicate symbol $\cO$. We will be working with expansions $\cK$ of an ordered or valued field $K$
to $\cL$, where the symbol $\cO$ is interpreted as explained above.
Definability refers to $\cL$ for subsets of $K^m$ and to $\Loag$ for
subsets of $|K^\times|$.

Now let $\CF(X)$ be the ring of continuous definable functions as explained in the introduction.
For all $f,g\in\CF(X)$
the sets $\{f=g\}$, $\{f\neq g\}$, $\{|f|\leq|g|\}$ and so on, are defined as the
subsets of $X$ on which the corresponding relation holds true. For
example $\{f=0\}=\{x\in X\tq f(x)=0\}$ is the zero-set of $f$.
On $\CF(X)$ we work with the relation
\[
  g\sqsubseteq f \iff \forall x\in X\ \big(g(x)=0\Rightarrow f(x)=0\big).
\]
We prove in Theorem~\ref{th:PosetZeroIsJac} that this relation is
definable in the ring $\CF(X)$.

\begin{definition}\label{de:vanish-on-sep-germs}
Let $X\subseteq K^m$ be a definable set, $p_0\in X$ and let $U$ be a definable
neighborhood of $p_0$ in $X$. We say that a function $s\in\CF(U)$
{\df vanishes on a germ} at $p_0$ if $s$ has a non-isolated zero at $p_0$.
We say that $s_1,\dots,s_k\in\CF(U)$  {\df vanish on separated germs}
at $p_0$ if there are
$\delta_1,\dots,\delta_k\in\CF(U\setminus\{p_0\})$ with the following properties.
\begin{description}
  \item[(S1)]
    The sets $S_i=\{s_i=0\}\setminus\{p_0\}$ are pairwise disjoint.
  \item[(S2)]
    Each $s_i$  vanishes on a germ at $p_0$.
  \item[(S3)]
    For $i\neq j$, the function $\delta_i$ is constantly $1$ on $S_i$ and constantly $0$ on $S_j$.
  \item[(S4)]
    Each $\delta_i$ is bounded on $U\setminus\{p_0\}$.
\end{description}
We call functions $\delta_1,\dots,\delta_k$ with these properties {\df separating functions}
for $s_1,\dots,s_k$. Finally we say that $X$ {\df has arbitrarily many
germs} at $p_0$ if for every positive integer $k$ there is a definable
neighborhood $U$ of $p_0$ in $X$ and $k$ functions in $\CF(U)$
that vanish on separated germs at $p_0$.
\end{definition}

\begin{examples}
  Intuitively this means that $p_0$ can be approached in $X$ through
  arbitrarily many disjoint ways.
  \begin{enumerate}
    \item
      For every $r\geq2$, $X=K^r$ has arbitrarily many germs at the
      origin (see Proposition~\ref{pr:mdir-open}).
    \item
      If $\cK$ is $o$\--minimal then $X=K$ does not have arbitrarily many
      germs at $0$, because any definable set $S\subseteq K\setminus\{0\}$ whose closure
      contains $0$ will necessarily meet one of the two intervals
      $(-\infty,0)$ or $(0,+\infty)$.
    \item
      If $\cK$ is any expansion of the real field $\RR$ including the $\sin$
      function then $X=K$ has arbitrarily many germs at $0$: take
      $s_i(x)=x\sin((1-ix)/(kx))$ for $1\leq i\leq k$.
    \item
      If $\cK$ is any expansion of a valued field whose value group is
      a $\ZZ$\--group then $K$ has arbitrarily many germs at $0$: take
      $s_i(x)=x\chi_i(x)$ where $\chi_i$ is the characteristic function of the
      (clopen) set of elements of $\cO\setminus\{0\}$ whose valuation is
      congruent to $i$ modulo $k$.
  \end{enumerate}
\end{examples}

\noindent
A {\df coordinate projection} $\pi:K^m\lra K^r$ is a
map of the form $\pi(x)=(x_i)_{i\in I}$ for some
$I\subseteq\{1,\dots,m\}$ of size $r\geq 0$; we write $\pi_I$ when necessary.
The {\df dimension} of a non-empty set $X$ is the maximal $r\leq m$
such that $\pi(X)$ has non-empty interior for some
coordinate projection $\pi:K^m\to K^r$.
This is extended to $\dim(\emptyset)=-\infty$.
The {\df local dimension} of $X$
at a point $x\in K^m$ is
\begin{displaymath}
  \dim(X,x)=\min\big\{\dim B\cap X\tq B\mbox{ is an open ball centered at
  }x\big\}.
\end{displaymath}
Note that $\dim(X,x)=-\infty$ if and only if $x\notin\overline{X}$, and
that $\dim(X,x)=0$ if and only if $x$ is an isolated point of $X$.

\begin{definition}\label{de:defnDim}
For every integer $d\geq 0$ we write
\begin{displaymath}
  \Delta_d(X)=\big\{x\in X\tq \dim(X,x)=d\big\}.
\end{displaymath}
Further we write $W_d(X)$ for the set of all $x\in X$ for which
there is an open ball $B$ centered at $x$ and a coordinate
projection $\pi:K^m\to K^d$ that induces by restriction a homeomorphism
between $B\cap X$ and an open subset of $K^d$.

Since the open balls are uniformly definable in $\Lring\cup\{\cO\}$ the sets
$\Delta_d(X)$ and $W_d(X)$ are definable in $\cL$.

We say that \textbf{$\cK$ satisfies \hypertarget{Dim}{(Dim)}} if for every
definable set $X\subseteq
K^m$ and every definable map $f:X\to K^r$ the following properties hold
true.

\begin{listD}
%   \item\label{it:dim-union}
%     $\dim X\cup Y=\max(\dim X,\dim Y)$.
  \item\label{it:dim-image}
    $\dim (f(X))\leq\dim (X)$.
  \item\label{it:dim-frontier}
    $\dim (X) =\dim(\overline{X})$ and if $X\neq \0$, then $\dim (\partial X)<\dim (X)$.
  \item\label{it:dim-Sk}
    If $\dim (X)=d\geq 0$ then $\dim (X\setminus W_d(X))<d$.\footnote{
In the classical cases of $o$\--minimal, $C$\--minimal and $P$\--minimal structures, the set $\Delta_d(X)$
is usually considered instead of
    $W_d(X)$ in (Dim\ref{it:dim-Sk}). However it is this slightly
    stronger statement with $W_d(X)$ which we need in
    Section~\ref{se:loc-dim}. It appears in Proposition~4.6 of
    \cite{simo-wals-2018}.}
\end{listD}
\end{definition}

\begin{remark}\label{re:dim-dp}
  These properties hold true in every dp-minimal expansion of a field which is not strongly
minimal (\cite{simo-wals-2018}). This implies and generalises
  known results  on $o$\--minimal, $C$\--minimal and $P$\--minimal fields (see
  \cite{drie-1998}, \cite{hask-macp-1994}, \cite{hask-macp-1997} and
  \cite{darn-cubi-leen-2017}).
\end{remark}

\noindent
Note that the sets $\Delta_d(X)$ are pairwise disjoint and that $\bigcup_{l\geq
d}\Delta_l(X)$ is closed in $X$ for each $d$, while $W_d(X)$ is open in
$X$.

\begin{property}\label{py:Sk-Dk}
  Property (Dim\ref{it:dim-image}) implies that the dimension is preserved by definable
  bijections and that
  $W_d(X)\subseteq\Delta_d(X)$ for every $d\geq0$. In particular the sets $W_d(X)$ are
  pairwise disjoint.
\end{property}

\begin{proof}
  The first assertion is obvious, we prove the second one. If $x\in
  W_d(X)$, then there is a definable neighborhood $U$ of $x$ in $X$, a
  coordinate projection $\pi:K^m\to K^d$ and an open subset $V$ of $K^d$
  such that $\pi|_U$ is a homeomorphism onto
  $V$. In particular $\dim U=d$ by the first assertion, hence $\dim
  W_d(X)\geq d$. For every sufficiently small open ball $B$ centered at $x$ we have
  $B\cap X\subseteq U$, hence $\pi|_{B\cap X}$ is a homeomorphism
  onto a non-empty open subset of $K^d$ and so $\dim (B\cap X)=d$. This proves
 $\dim(X,x)=d$ hence $W_d(X)\subseteq\Delta_d(X)$. Since the sets $\Delta_d(X)$
  are pairwise disjoint, so are the sets $W_d(X)$.
\end{proof}

\section{Definability of the poset of zero sets}
\label{se:SectionDefZeroPoset}

\noindent
Let $X$ be a set and let $A$ be a ring of functions $X\lra K$ for some field $K$.
We show in Theorem~\ref{th:PosetZeroIsJac} that for a huge class of examples, the $(n+1)$-ary relation $\{f_1=0\}\cap \ldots\cap \{f_n=0\}\subseteq \{g=0\}$ of $A$
is equivalent to $g$ being in the Jacobson radical of the ideal $(f_1,\ldots,f_n)$.
In particular, theses relations are 0-definable in the ring $A$.
The crucial ingredients are contained in Proposition~\ref{pr:Kuv} and Proposition~\ref{pr:Auv}.
\smallskip

\noindent
Let $I$ be an ideal of a ring $A$.
The \textbf{Jacobson radical} $\Jac(I)$ of $I$ is defined as the intersection of the maximal ideals of $A$
containing $I$ ({\it cf.} \cite[p. 3]{Matsum1989}).  The Jacobson radical of \textit{the ring} $A$ is defined as $\Jac(0)$.

\begin{remark}\label{re:RemindJac}
  One checks easily that
$\Jac(0)=\{a\in A\st\forall x\in A: 1+ax\in A^\times\}$, where $A^\times$ denotes the units of $A$.
Translating this description for $A/I$ back to $A$ shows that
\(
\Jac(I)=\{a\in A\st\forall x\in A\ \exists y\in A,z\in I: (1+ax)y=1+z\}.
\)
\end{remark}

\begin{proposition}\label{pr:Kuv}
Let $K$ be a field.
\begin{enumerate}

\item\label{it:KuvNacl}
If $K$ is not algebraically closed, then there are polynomials $u(x,y),v(x,y)$ in two variables over
$K$ such that the unique zero of $xu(x,y)+yv(x,y)$ in $K^2$ is $(0,0)$.

\item\label{it:KuvInvolution}
Assume that $R$ is a real closed field and $K=R[i]$ is its algebraic closure.
We write $a^*$ for the complex conjugation of $a\in K$ with respect to $R$.
Then for the functions $u(x,y)=x^*$ and $v(x,y)=y^*$ defined on $K^2$,
the unique zero of $xu(x,y)+yv(x,y)$ in $K^2$ is $(0,0)$.

\item\label{it:KuvACVF}
Assume that $K$ is a topological
    field (see \cite{warn-1993}) of characteristic $\neq 2$, where a basis of neighborhoods of $0\in K$
    is given by the non-zero ideals of a ring $\cO$ with fraction
    field $K$ and non-zero Jacobson radical\footnote{Our principal example here is a proper valuation ring $\cO$ of $K$.}.

Then there are $\tau$-continuous functions $u,v:K^2\lra K$ that are definable in the expansion  $(K,\cO)$
of $K$ by the set $\cO$,
such that the unique zero of $xu(x,y)+yv(x,y)$ in $K^2$ is $(0,0)$.
\end{enumerate}
\end{proposition}
\begin{proof}
\ref{it:KuvNacl} Since $K$ is not algebraically closed there is a polynomial
$p(x)=x^d+a_{d-1}x^{d-1}+\ldots+a_0\in K[t]$ without zeroes in $K$ such that $d\geq 1$.
Then its homogenization
\[
q(x,y)=x^d+a_{d-1}x^{d-1}y+\ldots+a_0y^{d}
\]
has a unique zero in $K^2$, namely $(0,0)$: If $b=0$, then $q(a,b)=0$ just if $a=0$.
If $b\neq 0$, then $q(a,b)=b^d\cdot p(\frac{a}{b})\neq 0$.

We choose $u(x,y)=x^{d-1}$ and $v(x,y)=a_{d-1}x^{d-1}+\ldots+a_0y^{d-1}$ and see that
$xu(x,y)+yv(x,y)=q(x,y)$ has the required properties.

\smallskip\noindent
\ref{it:KuvInvolution} is clear.

\smallskip\noindent
\ref{it:KuvACVF} We write $|K|=\{a\cO\st a\in K\}$ for the set of principal fractional ideals of $\cO$
and $|a|=a\cO$ for $a\in K$. Further we write $|a|\leq |b|$ instead of $|a|\subseteq |b|$.

\smallskip\noindent
\textit{Claim.}
For every topological space $X$, each continuous function
$f:X\lra K$ and all $x_0\in X$ with $f(x_0)\neq 0$ there is an open neighborhood
$U$ of $x_0$ in $X$ such that the restriction of $|f|:X\lra |K|$ to $U$ is constant.

\noindent
\textit{Proof.}
Replacing $f$ by $\frac{1}{f(x_0)}\mal f$ if necessary,  we may assume that $f(x_0)=1$.
By the assumption in (iii) there is a non-zero
element $\eps$ in the Jacobson radical of $\cO$
such that $1+\eps\cO$ is an open neighborhood of $1$ in $K$.
By continuity of $f$ at $x_0$ there is a neighborhood $U$
of $x_0$ in $X$ with $f(U)\subseteq 1+\eps\cO$.
Since $\eps\cO\subseteq \Jac(0)$, no element in $1+\eps\cO$ can be in any maximal ideal of $\cO$.
Hence $f(U)\subseteq 1+\eps\cO\subseteq A^\times$, which shows that
$|f(x)|=|1|$ for all $x\in U$.
\hfill$\diamond$

\smallskip\noindent
Returning to the proof of \ref{it:KuvACVF}, we define
\begin{displaymath}
    U=\big\{|x-y|<|x+y|\big\},\quad
    V=\big\{0\neq |x-y|\nless|x+y|\big\},
\end{displaymath}
\begin{displaymath}
    Z=\big\{x-y=x+y=0\big\}.
\end{displaymath}
This is a partition of $K^2$ in $(K,\cO)$-definable sets.
Note
that $x+y\neq 0$ in $U$ and  $x-y\neq 0$ in $V$.
Using the claim with $X=K^2$ it is easy to see that $U$ and $V$ are open in $K^2$.
\iflongversion\LongStart
This is clear for $V$ and also for $U$ when the base point is not on the diagonal.
Proof for $U$, the function $f(x,y)=x+y$
and a point $(x_0,x_0)\in U$. Then $2x_0\neq 0$ and after division by $2x_0$ we may assume that
$x_0=\frac{1}{2}$. Choose $\eps\in \Jac(0)$. Then $(\frac{1}{2}+\eps\cO)\times (\frac{1}{2}+\eps\cO)\subseteq U$.
\LongEnd\else\fi
Further, the
functions $u,v:K^2\lra K$ defined by
\[
u(x,y)=
  \begin{cases}
    x+y & \text{if }(x,y)\in U, \cr
    x-y & \text{if }(x,y)\in K^2\setminus U
  \end{cases}
\quad
\text{and}
\quad
v(x,y)=
  \begin{cases}
    x+y & \text{if }(x,y)\in U, \cr
    y-x & \text{if }(x,y)\in K^2\setminus U,
  \end{cases}
\]
are continuous on $U$ and on $V$.
Moreover, since $x-y=x+y=0$ on $Z$,
both $u$ and $v$ tend to $0$ at every point of $Z$.
Thus $u$ and $v$ are continuous on $X$. By
construction they are also definable in $(K,\cO)$. On $U$,
$xu(x,y)+yv(x,y)=(x+y)^2$ has no zero. On $V$,
$xu(x,y)+yv(x,y)=(x-y)^2$ has no zeroes.

Since $K$ has characteristic $\neq 2$, the set $Z$ is $\{(0,0)\}$, which  establishes the assertion.
\iflongversion\LongStart
In characteristic 2, the function vanishes on the diagonal $\subseteq K^2$.
\LongEnd\else\fi
\end{proof}

\begin{proposition}\label{pr:Auv}
Let $K$ be a field and let
$u,v:K^2\lra K$ be functions such that the unique zero of $xu(x,y)+yv(x,y)$ in $K^2$ is $(0,0)$.

Let $X$ be a set and let $A$ be a ring of functions $X\lra K$
such that $A$ is closed under composition
with $u$ and $v$, i.e., $u\circ(f,g),v\circ (f,g)\in A$ for all $f,g\in A$.

\begin{enumerate}
\item\label{it:AuvPrincipal}
For all $f_1,\ldots,f_n\in A$ there are $g_1,\ldots,g_n\in A$ with
\[
\{f_1=0\}\cap \ldots\cap\{f_n=0\}=\{g_1f_1+\ldots+g_nf_n=0\}.
\]

\item\label{it:AuvMain}
If
every $f\in A$ without zeroes in $X$ is a unit in $A$ then
for all $f_1,\dots,f_n,g\in A$ we have
\[
\{f_1=0\}\cap \ldots\cap \{f_n=0\}\subseteq \{g=0\}\ \Longrightarrow\ g\in\Jac(f_1,\ldots,f_n).
\]
\iflongversion\LongStart
Notice that (the contrapositive of) this implication is equivalent to saying that
$\Spec (A)^{\max}$ is included in the constructible closure of
the image $\hat X$ of the map $X\lra \Spec(A),\ x\mapsto \{f\in A\st f(x)=0\}$.
\LongEnd\else\fi

\item\label{it:AuvMinor}
If
$A$ contains all constant functions $X\lra K$, then
for all $f_1,\dots,f_n,g\in A$ we have
\[
g\in\Jac(f_1,\ldots,f_n)\ \Longrightarrow\ \{f_1=0\}\cap \ldots\cap \{f_n=0\}\subseteq \{g=0\}.
\]
\iflongversion\LongStart
Notice that (the contrapositive of) this implication is equivalent to saying that
$\hat X$ is included in the constructible closure of $\Spec (A)^{\max}$.
\LongEnd\else\fi

\end{enumerate}
\end{proposition}
\begin{proof}
\ref{it:AuvPrincipal}. By induction on $n$, where $n=1$ is trivial.
The induction step readily reduces to the claim in the case $n=2$.
By assumption on $A$ we know that $h:=f_1\cdot (u\circ(f_1,f_2))+f_2\cdot (v\circ (f_1,f_2))\in A$.
By assumption on $u,v$ we see that the zero set of $h$ in $X$
is $\{f_1=0\}\cap \{f_2=0\}$. Hence we may take $g_1=u\circ(f_1,f_2)$
and $g_2=v\circ(f_1,f_2)$.

\smallskip\noindent
\ref{it:AuvMain}.
Assume $\bigcap_{i=1}^n\{f_i=0\}\subseteq \{g=0\}$
and let $\Dm$ be a maximal ideal of $A$ containing $f_1,\ldots,f_n$.
We need to show that $g\in \Dm$. By \ref{it:AuvPrincipal}, there is some $h$ in the ideal
$(f_1,\ldots,f_n)$
generated by $f_1,\ldots,f_n$ in $A$
with zero set $\bigcap_{i=1}^n\{f_i=0\}$.
In particular $h\in\Dm$ and $\{h=0\}\subseteq \{g=0\}$.

Suppose $g\notin\Dm$. Then there is some $a\in A$ and some $m\in\Dm$ with $ag+m=1$,
in particular $\{g=0\}\cap \{m=0\}=\0$. Since $\{h=0\}\subseteq \{g=0\}$
we get $\{h=0\}\cap \{m=0\}=\0$.
By \ref{it:AuvPrincipal} again, there are $b_1,b_2\in A$ such that $b_1h+b_2m$ has no zeroes.
But then by assumption on $A$, $b_1h+b_2m$ is a unit of $A$, contradicting
$h,m\in\Dm$.

\smallskip\noindent
\ref{it:AuvMinor}.
Let $x\in \bigcap_{i=1}^n\{f_i=0\}$ and let $e:A\lra K$ be the evaluation map at $x$.
Since $A$ contains the constant functions we know that $e$ is surjective, hence $\Dm=\ker(e)$ is a maximal ideal
of $A$. Then $f_1,\ldots,f_n\in \Dm$ and so by assumption $g\in\Dm$, i.e. $g(x)=0$.
\end{proof}

\begin{theorem}\label{th:PosetZeroIsJac}
Let $X$ be a set and let $A$ be a ring of functions $X\lra K$
containing the constant functions
such that every $f\in A$ without zeroes in $X$ is a unit in $A$.
Suppose one of the following conditions hold:
\begin{enumerate}[(a)]
\item $K$ is not algebraically closed, or,
\item
$K$ is the algebraic closure of a real closed field $R$ and $A$ is closed under
conjugation of $K=R[i]$, or,
\item $K$ is a topological
    field  of characteristic $\neq 2$, where a basis of neighborhoods of $0\in K$
    is given by the non-zero ideals of a ring $\cO$ with fraction
    field $K$ and non-zero Jacobson radical.
    Further assume that $w\circ (f,g)\in A$ for every
    $(K,\cO)$-definable continuous function $w:K^2\lra K$.
\end{enumerate}
The following conditions are equivalent for all $f_1,\dots,f_n,g\in A$.
\begin{enumerate}
\item\label{it:PosetZeroIsJacZ}
$\{f_1=0\}\cap \ldots\cap \{f_n=0\}\subseteq \{g=0\}\iff g\in\Jac(f_1,\ldots,f_n).$
\item\label{it:PosetZeroIsJacJ}
$g\in\Jac(f_1,\ldots,f_n)$.
\item\label{it:PosetZeroIsJacAx}
For all $h\in A$, the element $1+hg$ is a unit modulo the ideal $(f_1,\ldots,f_n)$.
\end{enumerate}

\noindent
Consequently, the $(n+1)$-ary relation $\{f_1=0\}\cap \ldots\cap \{f_n=0\}\subseteq \{g=0\}$ of $A$
is definable in the ring $A$ by the $\Lring$-formula
\[
\forall x\, \exists y_1,\ldots,y_n,z\ (1+x\mal g)\mal z=1+y_1f_1+\ldots+y_nf_n.
\]
Of particular interest for us is the case $n=1$. Hence the binary relation
$f\sqsubseteq g$ defined as $\{f=0\}\subseteq \{g=0\}$, is 0-definable in $A$.
\end{theorem}
\begin{proof}
The assumptions in Proposition~\ref{pr:Auv} hold by Proposition~\ref{pr:Kuv}. Hence
Proposition~\ref{pr:Auv}\ref{it:AuvMain},\ref{it:AuvMinor} imply the equivalence of \ref{it:PosetZeroIsJacZ} and \ref{it:PosetZeroIsJacJ}.
The equivalence of \ref{it:PosetZeroIsJacJ} and
\ref{it:PosetZeroIsJacAx} holds by Remark~\ref{re:RemindJac}.
\end{proof}

\begin{examples}
Let $X$ be a topological space. Then Theorem~\ref{th:PosetZeroIsJac} applies to the following rings $A$ of functions $X\lra K$.

\begin{enumerate}[(a)]
\item $K$ is an ordered field or a $p$\--valued field and
\begin{itemize}
\item $A$ is the ring of continuous functions $X\lra K$, or,
\item $X\subseteq K^n$ is definable in $K$ and $A$ is the ring of definable continuous functions $X\lra K$, or,
\item $X\subseteq K^n$ is open (definable) and $A$ is the ring of (definable) $k$-times differentiable functions $X\lra K$, or,
\item $X\subseteq K^n$ is a variety and $A$ is the ring of rational functions $X\lra K$ without
zeroes on $X$ (sometimes referred to as \textit{regular functions} in the literature).
\end{itemize}
In each case, condition (a) of Theorem~\ref{th:PosetZeroIsJac} applies.

\item $A$ is the ring of continuous functions $X\lra \C$, or, $X\subseteq \C^n$ and
$A$ is the ring of continuous semi-algebraic functions $X\lra \C$.
In both cases, condition (b) of Theorem~\ref{th:PosetZeroIsJac} applies.

\item $K$ is a valued field of characteristic $\neq 2$, furnished
with the valuation topology and
$A$ is the ring of continuous functions $X\lra K$, or,
$X\subseteq K^n$ is definable and
$A$ is the ring of definable continuous functions $X\lra K$.
In both cases, condition (c) of Theorem~\ref{th:PosetZeroIsJac} applies, where $\cO$ is the valuation ring of $K$.
\end{enumerate}
\end{examples}

\section{Basic properties of zero sets}
\label{se:prelim}

We collect a few basic facts, which will be used in the rest
of the paper. We continue to work with the set up of the introduction and section
\ref{se:notation}.

\begin{lemma}\label{le:dirac}
\begin{enumerate}
\item\label{it:def-norm}
There is $\nu_m\in\CF(K^m)$ such that
  $|\nu_m(x)|=\|x\|$ for every $x\in K^m$.
\end{enumerate}

\smallskip\noindent
Now let $X\subseteq K^m$ be a definable set, $a\in X$ and
 let $B$ be a closed ball (resp. $B_0$ an open ball) with radius $r\in K^\times$
(resp. $r_0\in K^\times$) and center $a$.
\begin{enumerate}[resume]
  \item\label{it:dirac-pt}\label{it:dirac-ball}\label{it:dirac-co-ball}
  Each of the sets $\{a\}$, $B$ and $B_0^c=K^m\setminus B_0$ are zero sets of functions from $C(X)$. We pick such functions and denote them by $\delta_a$, $\delta_B$   and $\delta_{B_0^c}$ respectively.
\ifoldversion\OldStart
    \item%\label{it:dirac-pt}
      There is a function $\delta_a\in\CF(X)$ with zero set $\{x\}$.
    \item%\label{it:dirac-ball}
      There is a function $\delta_B\in\CF(X)$ with zero set $B$.
    \item%\label{it:dirac-co-ball}
      There is a function $\delta_{B_0^c}\in\CF(X)$ such that
      $\delta_{B_0^c}(x)=0\iff x\notin B_0$.\OldEnd\else\fi

    \item\label{it:dirac-B-B0}
      If $B\subseteq B_0$, then there is a function $\delta_{B,B_0}\in\CF(X)$ with values
      in $\cO$ such that $\delta_{B,B_0}$ vanishes on $B_0^c$ and
      $\{\delta_{B,B_0}=1\}=B$.
\ifoldversion\OldStart
$\delta_{B,B_0}(x)=1$ on $B$ and $\delta_{B,B_0}(x)=0\iff x\in B$.
\OldEnd\else\fi
\end{enumerate}
\end{lemma}

\begin{proof}
\ref{it:def-norm}.
  For every $x=(x_1,\dots,x_m)$ let $\nu_m(x)=x_1$ if $\|x\|=|x_1|$,
  $\nu_m(x)=x_2$ if $\|x\|=|x_2|>|x_1|$, and so on. It is obviously
  definable, and easily seen to be continuous on $K^m$.

\smallskip\noindent
\ref{it:dirac-pt} and \ref{it:dirac-B-B0}.
We may take $\delta_a(x)=\nu_m(x-a)$ restricted to $X$, using \ref{it:def-norm}. For the other functions, in the valued case we may take the indicator
  function of $X\setminus B$ for $\delta_B$, and the indicator function of $B_0$
  for $\delta_{B_0^c}$ as well as for $\delta_{B,B_0}$ (they are continuous because $B$  and $B_0$ are clopen).   In the ordered case we have $0<r_0<r$ and $|K|$
  identifies with the set of non-negative elements of $K$ so we may take
  $\delta_B(x)=\max\{0,\,\|x-a\|-|r|\}$, $\delta_{B_0^c}(x)=\max\{0,r-\|x-a\|\}$ and
  $\delta_{B,B_0}(x)=\max\{0,\min\{1,u(x)\}\}$ where $u(x)=(r-\|x-a\|)/(r-r_0)$.
\ifoldversion\OldStart   The first point follows from \ref{it:def-norm} with
  $\delta_a(x)=\nu_m(x-a)$ restricted to $X$.
  For the last points, in the valued case we can take the indicator
  function of $X\setminus B$ for $\delta_B$, and the indicator function of $B_0$
  for $\delta_{B_0^c}$ and $\delta_{B,B_0}$ (they are continuous because $B$
  and $B_0$ are clopen). In the ordered case $0<r_0<r$ and $|K|$
  identifies to the set of non-negative elements of $K$ so we can take
  $\delta_B(x)=\max\{0,\,\|x-a\|-|r|\}$, $\delta_{B_0^c}(x)=\max\{0,r-\|x-a\|\}$ and
  $\delta_{B,B_0}(x)=\max\{0,\min\{1,u(x)\}\}$ where $u(x)=(r-\|x-a\|)/(r-r_0)$.
\OldEnd\else\fi
\end{proof}

\begin{lemma}\label{le:PtIntIsol}
  There are formulas $\point(f)$ and $\inter(f,g,h)$ in $\Lring$,
  and a formula $\isol(s,p)$ in $\Lring\cup\{\cB\}$ such that for every
  definable set $X\subseteq K^m$:
  \begin{enumerate}
    \item\label{it:PtIntIsolPt}
      $\CF(X)\models\point(f)\iff f$ has a single zero in $X$.
    \item\label{it:PtIntIsolInt}
      $\CF(X)\models\inter(f,g,h)\iff\{f=0\}\cap\{g=0\}=\{h=0\}$.
    \item\label{it:PtIntIsolIsol}
      $(\CF(X),\cB)\models\isol(s,p)\iff\{p=0\}$ is an isolated point of
      $\{s=0\}$.
  \end{enumerate}
\end{lemma}

\begin{proof}
  Let $\displaystyle\ZS(X)=\big\{\{s=0\}\tq s\in\CF(X)\big\}$ ordered by
  inclusion. It is a bounded distributive lattice. Indeed for every
  $f,g\in\CF(X)$ we have
\begin{align*}
    \{f=0\}\cup\{g=0\}&=\{fg=0\}\cr
    \{f=0\}\cap\{g=0\}&=\{\nu_2(f,g)=0\},
\end{align*}
  where $\nu_2\in\CF(K^2)$ is the function given by Lemma~\ref{le:dirac} \ref{it:def-norm}.
  For every $a\in K^m$, $\{a\}=\{\delta_a=0\}$ where $\delta_a$ is given by
  Lemma~\ref{le:dirac} \ref{it:dirac-pt}, so the atoms of $(\ZS(X),\subseteq)$ are exactly the
  singletons. The two first points then follow from the fact that
  $(\ZS(X),\subseteq)$ is uniformly interpretable in $\CF(X)$ as the
  quotient ordered set of the preorder $\sqsubseteq$, see Theorem~\ref{th:PosetZeroIsJac}.

  For item \ref{it:PtIntIsolIsol}, assume that $\{p=0\}=\{p_0\}$. If $p_0$ is
  an isolated point of $\{s=0\}$, let $B_0$ be an open ball with center
  $p_0$ which is disjoint from $S=\{s=0\}\setminus\{p_0\}$.
  By Lemma~\ref{le:dirac} \ref{it:dirac-pt} there is a function $\delta_{B_0^c}\in\CF(X)$ such that $\delta_{B_0^c}(x)=0$
  if and only if $x\notin B_0$. Finally let $u=\nu_2(s,\delta_{B_0^c})$ where $\nu_2$ is
  given by Lemma~\ref{le:dirac} \ref{it:def-norm}. Then $\{u=0\}=\{s=0\}\cap B_0^c=S$, hence $s\sqsubseteq
  pu$ and $s\not\sqsubseteq u$. Conversely if there is $u\in\CF(X)$ such that $s\sqsubseteq
  pu$ and $s\not\sqsubseteq u$ then $\{u\neq0\}$ is a neighborhood of $p_0$ disjoint
  from $S$ hence $p_0$ is an isolated point $\{s=0\}$. So this property
  is axiomatized by the formula
\[
\isol(s,p)\equiv\point(p)\land p\sqsubseteq s \land \exists u\big(s\sqsubseteq pu\land s\not\sqsubseteq u\big).
\]
\end{proof}

\begin{proposition}\label{pr:local-formulas}
  For every parameter-free formula $\varphi(y)$ in $\Lring$ (resp.
  $\Lring\cup\{\cO\}$) in $k$ free variables, there is a parameter-free
  formula $[\varphi]$ in $\Lring$ (resp. $\Lring\cup\{\cB\}$) in $k+1$ variables
such that for every definable set $X\subseteq K^m$,
  every $h\in\CF(X)^k$ and every $s\in\CF(X)$ we have
\[
    (\CF(X),\cB)\models[\varphi](h,s)\iff \forall x\in\{s=0\},\  K \models\varphi(h(x)).
\]
\end{proposition}

\begin{proof}
  If $\varphi(y)$ is a polynomial equation $P(y)=0$ ($y$ a single
  variable), we may take
  $[\varphi](y,s)$
  as $s\sqsubseteq P(y)$. If $\varphi(y)$ is the formula $P(y)\in\cO$,  we may
  take $[\varphi](y,s)$ as
\[
    \forall p\,\bigg(\point(p)\land p\sqsubseteq s\ \longrightarrow\ \exists f\,\big(
    f\in\cB\land p\sqsubseteq P(y)-f\big)\bigg),
\]
expressing the fact that a function $P(y)$ has values in $\cO$ on the
zero set of $s$ just if for all $x\in \{s=0\}$, $P(y)$ agrees with a
function $f\in\cB$ at $x$.

\smallskip\noindent
This proves the result for atomic formulas. For arbitrary formulas,
  $[\varphi]$ is defined by induction as follows.
  \begin{itemize}
    \item
      $[\varphi\land\psi]$ is $[\varphi]\land[\psi]$.
    \item
      $[\lnot\varphi](y,s)$ is
      $\forall p\, \big((\point(p)\land p\sqsubseteq s)\rightarrow \lnot[\varphi](y,p)\big)$.
    \item
      $[\exists w\,\varphi](y,s)$ is $\forall p\, \big((\point(p)\land p\sqsubseteq s)\rightarrow
      \exists w\,[\varphi](w,y,p)\big)$.
  \end{itemize}
  We only show the right to left implication of the claimed equivalence in the case of an existential quantifier,
  all other implications are straightforward. Take
  $h\in\CF(X)^k$, $s\in\CF(X)$ and assume that $K\models\exists w\,\varphi(w,h(x))$ for
  every $x\in\{s=0\}$. We need to show that
  $(\CF(X),\cB)\models\forall p\, \big((\point(p)\land p\sqsubseteq s)\rightarrow
      \exists w\,[\varphi](w,y,p)\big)$.
  So take $p\in\CF(X)$ whose unique zero $p_0$ satisfies $s(p_0)=0$.
  By assumption there is some $c\in K$ with $K\models \varphi(c,h(p_0))$.
  Let $h_0\in \CF(X)$ be the constant function with value $c$.
  Now, for every $z\in X$ and every $q\in\CF(X)$ that has a unique zero $z$,
  if $p(z)=0$ then  $K\models \varphi(h_0(z),h(z))$. By induction, this means
  $(\CF(X),\cB)\models [\varphi](h_0,h,p)$, as required.

\smallskip
Finally, by choice of the formulas $[\varphi]$ we see that $[\varphi]$ is an $\Lring$-formula, if
$\phi$ is an $\Lring$-formula.
\end{proof}

\begin{remark}\label{re:def-O3}
  If $\cO$ is definable in $K$ by a formula $\varphi(x, a)$ in $\Lring$,
  where $a$ is an $n$\--tuple of parameters from $K$, then Proposition~\ref{pr:local-formulas} applied to $\varphi(x, y)$ implies that
  $\cB$ is definable in $\CF(X)$  by the
$\Lring$-formula $[\varphi](x,a^*,0)$
with parameters $a^*$, the constant function of $\CF(X)^n$ with value $a$.
  There are plenty of such fields with definable orders or valuations
  (including all the non-algebraically closed local fields):
  \begin{itemize}[itemsep=0pt]
    \item
      real-closed and $p$\--adically closed fields;
    \item
      the field of rational numbers (by Lagrange's Four Squares Theorem);
    \item
      one-variable functions fields over number fields
      \cite{mill-shla-2017};
    \item
      dp-minimal valued fields that are not algebraically closed
      \cite{john-2015};
    \item
      the field of Laurent series $F((t))$ with the natural valuation
      \cite{ax-1965};
  \end{itemize}
  and numerous others: see \cite{fehm-jahn-2017} for a summary on
  Henselian valuation rings definable in their fraction field.
\end{remark}

Having arbitrarily many germs at a point is a local property, but it
can be made a bit more global as follows.

\begin{property}\label{py:mdir-global}
  Let $X\subseteq K^m$ be a definable set and let $p_0\in X$. Then $X$ has
  arbitrarily many germs at $p_0$ (see Definition~\ref{de:vanish-on-sep-germs}) if and only if for every positive
  integer $k$ there exist $k$ functions in $\CF(X)$ that vanish on
  separated germs at $p_0$ with separating functions in
  $\CF(X\setminus\{p_0\})$.
\end{property}

\begin{proof}
  One implication is obvious. For the converse, assume that a
  definable neighborhood $U$ of $p_0$ in $X$ is given together with
  $v_1,\dots,v_k\in\CF(U)$ that vanish on separated germs at $p_0$, and
  with corresponding separating functions $d_1,\dots,d_k\in\CF(U\setminus\{p_0\})$.
  Restricting $U$ if necessary, by continuity at $p_0$, we may assume
  that each $v_i$ is bounded on $U$ and that $U=B_0\cap X$ for some open
  ball $B_0$ with center $p_0$. Let $B$ be a closed ball with center
  $p_0$ contained in $B_0$.
  By Lemma~\ref{le:dirac} \ref{it:dirac-B-B0} we have a bounded function $h=\delta_{B,B_0}\in\CF(X)$
  with $h(x)=0$ on $X\setminus B_0$ and $h(x)=1$ on $B\cap X$.

  For each $i\leq k$ let $u_i(x)=0$ on $X\setminus B_0$ and $u_i(x)=h(x)v_i(x)$
  on $B_0\cap X$. Similarly let $\delta_i(x)=0$ on $X\setminus B_0$ and
  $\delta_i(x)=h(x)d_i(x)$ on $B_0\cap X\setminus\{p_0\}$. Each $u_i$ (resp. $\delta_i$) is
  continuous on $B_0\cap X$ and tends to $0$ at every point in $\partial(B_0\cap
  X)$ (because $v_i$ and $\delta_i$ are bounded) hence $u_i\in\CF(X)$ and
  $\delta_i\in\CF(X\setminus\{p_0\})$. Now let $s_i=\nu_2(\delta_B,u_i)$, where
  $\nu_2,\delta_B\in\CF(K^2)$ are given by Lemma~\ref{le:dirac} \ref{it:def-norm},\ref{it:dirac-B-B0}. By construction $\delta_i$ is bounded on $X\setminus\{p_0\}$,
  $\delta_i(x)=d_i(x)$ on $B\cap X\setminus\{p_0\}$ and $\{s_i=0\}$ equals $\{v_i=0\}\cap B$.
  So the functions $s_i$, $\delta_i$ inherit from $v_i$, $d_i$ all the
  properties (S1)--(S4) of functions vanishing on
  separated germs, {\it cf}. Definition~\ref{de:vanish-on-sep-germs}.
\end{proof}

\section{Constructing integers using limit values and chunks}
\label{se:lim-chunk}

The present section is devoted to our first main results of
interpretability and definability. Our construction is based on the
following subset of $K$. Let $X\subseteq K^m$ be a definable set. Let $s,p\in\CF(X)$ be
such that the zero-set of $p$ is a single point $p_0$ and $s$ vanishes
on a germ at $p_0$. Then for all $f,g\in\CF(X)$ for which
$g$ has no zeroes in the set
$S=\{s=0\}\setminus\{0\}$, we consider the continuous function $f/g$ on $S$ and write
$\Gamma$ for its graph. We define
\[
L_{s,p}(f/g)=\big\{l\in K\tq (p_0,l)\in\overline{\Gamma}\big\}.
\]

\noindent
Informally, $L_{s,p}$ is the set of ``limit values'' at $p_0$ of $f/g$
restricted to $S$. Note that this is always a closed definable subset
of $K$.
We will also consider the following relations on $\CF(X)$. Note that
they are definable in $\Lring\cup\{\cB\}$ using Proposition~\ref{pr:local-formulas}.
\begin{displaymath}
  |f|\leq_s|g| \iff \{s=0\}\subseteq\{|f|\leq|g|\}
\end{displaymath}
\begin{displaymath}
  |f|<_s|g| \iff \{s=0\}\subseteq\{|f|<|g|\}.
\end{displaymath}

\begin{lemma}\label{le:lim-form}
  There is a parameter-free formula $\limit(f,g,h,s,p)$ in
  $\Lring$ such that for any definable set $X\subseteq K^m$, we have
  $\CF(X)\models\limit(f,g,h,s,p)$ if and only if
  \begin{itemize}[itemsep=0pt]
  \item   the zero-set of $p$ is a single point $p_0$,
  \item $s$ vanishes on a germ at $p_0$,
 \item  $g$ has no zeroes in the set $S=\{s=0\}\setminus\{p_0\}$ and
 \item $h(p_0)\in L_{s,p}(f/g)$.
  \end{itemize}

\end{lemma}

\begin{proof}
  The first three properties are defined by the conjunction
  $\chi(g,s,p)$
  of the formulas in Lemma~\ref{le:PtIntIsol}. Hence we define
  $\chi(g,s,p)$ as
  \begin{displaymath}
    \point(p)\land \lnot\isol(s,p) \land
    \big(\inter(g,s,p)\lor \inter(g,s,1)\big).
  \end{displaymath}
  Let $f,g,h,s,p\in\CF(X)$ be such that $\CF(X)\models\chi(g,s,p)$ and let
  $p_0$ be the zero of $p$. Pick $\tau\in K^\times$ with $|\tau|<1$. For
  every $\varepsilon,v\in\CF(X)$ with $\varepsilon(p_0)\neq0$ and $v(p_0)\neq0$, by  continuity, there is an open ball $B_0$ centered at $p_0$ such that
  $B_0\cap X\subseteq\{v\neq0\}$ and for every $x\in B_0\cap X$, $|\varepsilon(x)|>|\tau\varepsilon(p_0)|$ and
  $|h(x)-h(p_0)|<|\tau\varepsilon(p_0)|$. So if $h(p_0)\in L_{s,p}(f/g)$, then there is a
  point $q_0\in B_0\cap X\setminus\{p_0\}$ with
  \begin{displaymath}
    s(q_0)=0\mbox{ and }
    \left|\frac{f(q_0)}{g(q_0)}-h(q_0)\right|\leq\big|\varepsilon(q_0)\big|.
  \end{displaymath}
  Note that $q_0\in B_0\setminus\{p_0\}$ implies that $q\not\sqsubseteq pv$, and that $s(q_0)=0$
  implies $g(q_0)\neq0$, so the second condition above is equivalent
  to $|f(q_0)-g(q_0)h(q_0)|\leq|g(q_0)\varepsilon(q_0)|$.
  Conversely, if there are such points $q_0$ for any $\varepsilon,v\in\CF(X)$
  with $\varepsilon(p_0)\neq0$ and $v(p_0)\neq0$, then setting $v=\delta_{B_0^c}$
  for any open ball $B_0$ centered at $p_0$ (with $\delta_{B_0^c}$ given by
  Lemma~\ref{le:dirac}\ref{it:dirac-co-ball}), we get that
  $h(p_0)\in L_{s,p}(f/g)$. Thus we can take for $\limit(f,g,h,s,p)$ the
  conjunction of $\chi(g,s,p)$ and the formula
  \begin{displaymath}
    \forall \varepsilon,\, v\, \big(p\not\sqsubseteq v\varepsilon \to
    \exists q\,\big[ \point(q)\land q\not\sqsubseteq pv \land q\sqsubseteq s\land |f-gh|\leq_q|g\varepsilon|\big]\big).
  \end{displaymath}
\end{proof}

\begin{lemma}\label{le:limits}
  Let $l_1,\dots,l_k\in K$ and let $X\subseteq K^m$ be a definable set. Let $p\in X$ be such
  that $p$ has a single zero $p_0$ at which $X$ has arbitrarily many
  germs.
  Let $s_1,\dots,s_k\in\CF(X)$ be vanishing on separated germs at $p_0$ and let
  $S_i=\{s_i=0\}\setminus\{p_0\}$. Then there are $f,g\in\CF(X)$ with
  $g(x)\neq0$ on $S=S_1\cup\cdots\cup S_k$ such that the restriction of $f/g$ to each
  $S_i$ has constant value $l_i$. In particular:
  \begin{enumerate}
    \item\label{it:limits-Lspfg}
      $\{l_1,\dots,l_k\}=L_{s,p}(f/g)$, with $s=s_1s_2\cdots s_k$.
    \item\label{it:limits-h}
      $\displaystyle\{l_1,\dots,l_k\}=\big\{l(p_0)\tq l\in\CF(X)$ and
      $p_0\in\overline{S\cap\{f=gl\}}\big\}$.
  \end{enumerate}
\end{lemma}
\begin{proof}
Let $\delta_1,\dots,\delta_s\in\CF(X\setminus\{p_0\})$ be separating functions for $s_1,\dots,s_k$.
Clearly $f=\sum_{i\leq k}l_i\delta_ip$ (extended by $0$ at $p_0$) and $g=p$ have
the required properties, from which items~\ref{it:limits-Lspfg} and
\ref{it:limits-h} follow immediately. For the second item, note that
$p_0\in\overline{S\cap\{f=gl\}}$ simply means that the function $l$ takes the
same values as $f/g$ on a subset of $S$ having points arbitrarily
close to $p_0$, hence $l(p_0)\in L_{s,p}(f/g)$ by continuity of $l$.
\end{proof}

\begin{definition}\label{de:defnChunk}
Let $(G,+,\leq)$ be  totally pre-ordered abelian group,
hence $\leq $ is a total pre-order satisfying
$x\leq y\Rightarrow x+z\leq y+z$.
For $\tau\in G$ with $\tau>0$ we write
$\tau\ZZ=\{\tau k\tq k\in\ZZ\}$.

We call a subset $T$ of $G$ a {\df $\tau\ZZ$\--chunk} of $G$ if $\tau\in T$ and for all $\alpha,\beta,\gamma\in T$:
\begin{enumerate}
  \item
    $-\alpha\in T$;
  \item
    $\alpha+\beta\in[-\gamma,\gamma]\Rightarrow \alpha+\beta\in T$;
  \item
    $\forall u\in[-\gamma,\gamma]$, $\exists!\,\xi\in T,\ \xi\leq u< \xi+\tau$.
\end{enumerate}
This should be seen as a finitary, and hence definable, version of integer parts as studied by \cite{mourg-ress-1993} in the case of real closed fields.
\end{definition}

\begin{remark}\label{re:fin-chunk}
 For every finite $\tau\ZZ$\--chunk of $G$ one checks easily that there is some integer $n$ with
$G=\{-n\tau,\dots,-\tau,0,\tau,2\tau,\dots,n\tau\}$.
\end{remark}
%Let $n\in\NN $ with $G\cap \tau\ZZ=\{-n\tau,\dots,-\tau,0,\tau,2\tau,\dots,n\tau\}$.

\noindent
In the ordered case (and more generally when there is a total pre-order
on $K$, definable in $\Lring\cup\{\cO\}$ and compatible with the group
structure of $(K,+)$), we can consider
chunks of $(K,+,\leq)$. In the valued case however we have to lift
 chunks from  $|K^\times|$ to $K^\times$. So we will consider chunks of
the totally pre-ordered multiplicative group $(K^\times,\times,\leq_{\cO})$
where $\leq_{\cO}$ is the inverse image of the order of $|K^\times|$, i.e.,
$x \leq_{\cO}y$ is just $|x|\leq|y|$.
In this case we adapt the terminology to multiplicative language. For example, for $\tau\in K^\times$ with
$|\tau|<1$ we write $\tau^\ZZ=\{\tau^k\tq k\in\ZZ\}$ instead of $\tau\ZZ$.
\ifoldversion\OldStart
set
\begin{displaymath}
\tau^\ZZ=\{\tau^k\tq k\in\ZZ\}
\end{displaymath}
and we say {\df $\tau^\ZZ$\--chunks} of $K^\times$
for any subset $T$ of $K^\times$ with $\tau\in T$ such that for every
$\alpha,\beta,\gamma\in T$:
\begin{enumerate}
  \item
    $1/\alpha\in T$,
  \item
    $|1/\gamma|\leq|\alpha\beta|\leq|\gamma|\Rightarrow \alpha\beta\in T$ and
  \item
    $\forall u\in K$, $|1/\gamma|\leq|u|\leq|\gamma|\Rightarrow\exists!\,\xi\in T,\ |\xi\tau|<|u|\leq|\xi|$.
\end{enumerate}
\OldEnd\else\fi

\begin{theorem}\label{th:def-chunk}
  Assume that $\cK$ satisfies \hyperlink{BFin}{(BFin)}. Let $I=\{x\in K\tq |0|<|x|<|1|\}$.
  There is a parameter-free formula $\mInt(p,h,\tau^*)$ in
  $\Lring\cup\{\cB\}$ such that for every definable set $X\subseteq K^m$, if
  $p$ has a unique zero $p_0\in X$, if $X$ has arbitrarily many germs at $p_0$
  and if the value $\tau=\tau^*(p_0)$ is in $I$, then
  \begin{displaymath}
    (\CF(X),\cB)\models\mInt(h,p,\tau^*)\iff
    h(p_0)\in\tau^\ZZ.
  \end{displaymath}
  If $K$ has a total order $\preccurlyeq$ compatible with $+$ and definable%
  \footnote{This obviously happens in the ordered case, where $K$ is an
    ordered field and $\cO$ is the interval $[-1,1]$. But
    Theorem~\ref{th:def-chunk} does not restrict to this case, as it
    does not assume any relation at all between $\cO$ and the order
    $\preccurlyeq$.}
  in $\Lring\cup\{\cO\}$, then there is similarly a formula
  $\aInt(h,p,\tau^*)$ such that, with the same assumptions on $p$, if
  $\tau=\tau^*(p_0)\succ0$ then
  \begin{displaymath}
    (\CF(X),\cB)\models\aInt(h,p,\tau^*)\iff h(p_0)\in\tau\ZZ.
  \end{displaymath}
\end{theorem}
\noindent
In other words, Theorem~\ref{th:def-chunk} says essentially that,
given any $\tau\in
K^\times$ with $|\tau|<1$ and any definable set $X\subseteq K^m$ with
arbitrarily many germs at some point $p_0$, then the set of all $h\in\CF(X)$ taking
values in $\tau^\ZZ$ at $p_0$, is definable in $(\CF(X),\cB)$;
the definition is independent of $X$, $K$ and $m$ and
is in the two definable parameters $p_0,\tau$.

Also notice that by Remark~\ref{re:def-O3}, the set $\cB$ is definable in $\Lring$ in many cases, hence Theorem~\ref{th:def-chunk} applies to
the pure ring $\CF(X)$ in all these cases.

\begin{proof}
  In order to ease the notation, in this proof we write
  $\CF(X)\models\chi(\dots)$ instead of $(\CF(X),\cB)\models\chi(\dots)$, for every formula
  in $\Lring\cup\{\cB\}$.

  We first construct a formula which axiomatizes
  (uniformly) the property that $\tau=\tau^*(p_0)\in\cO\setminus\{0\}$ and $L_{s,p}(f/g)$ is a
  bounded $\tau^\ZZ$\--chunk. We use the formula $\limit(f,g,h,s,p)$
  stated in Lemma~\ref{le:lim-form}. It will be convenient to
  abbreviate it as $\varphi^\sigma(h)$ where $\sigma=(f,g,s,p)$. Thus
  for every $f,g,v,s,p\in\CF(X)$ we have
\[
    v(p_0)\in L_{s,p}(f/g)\iff\CF(X)\models\varphi^\sigma(v).
\]
  Let $\psi^\sigma(\alpha,\alpha',\beta,\gamma,p,\tau^*)$ be the conjunction of the following formulas.
  \begin{enumerate}[1.]
    \item
      $p\sqsubseteq \alpha\alpha'-1 \to \varphi^\sigma(\alpha')$.
    \item
      $|1|\leq_p|\alpha\beta\gamma|\leq_p|\gamma^2|$.
    \item
      $\displaystyle \forall u\,\big(|1|\leq_p|u\gamma|\leq_p|\gamma^2|\to
      \exists\xi\,\big[\varphi^\sigma(\xi)\land |\xi\tau^*|<_p|u|\leq_p|\xi|\big]\big)$.
    \item
      $\displaystyle \forall u,\xi,\xi'\,
        \big(\big[ A(u,\xi,p,\tau^*)\land A(u,\xi',p,\tau^*) \big]
        \to p\sqsubseteq \xi-\xi'\big)$,\\
        where $A(u,\xi,p,\tau^*)$ stands for $\varphi^\sigma(\xi)\land |\xi\tau^*|<_p|u|\leq_p|\xi|$.
  \end{enumerate}
  Clearly $\tau=\tau^*(p_0)\in I$ and $L_{s,p}(f/g)$ is a $\tau^\ZZ$\--chunk if and
  only if $\CF(X)$ satisfies the conjunction, which we will denote
  $\mlimch(f,g,s,p,\tau^*)$, of $|0|<_p|\tau^*|<_p|1|$, of $\varphi^\sigma(\tau^*)$ and of
  \begin{displaymath}
    \forall \alpha,\alpha',\beta,\gamma\,
    \big[\varphi^\sigma(\alpha)\land \varphi^\sigma(\beta)\land \varphi^\sigma(\gamma) \to \psi^\sigma(\alpha,\alpha',\beta,\gamma,p,\tau^*)\big].
  \end{displaymath}
  The formula $\mlimBch(f,g,s,p,\tau^*)$ is then defined as
  \begin{displaymath}
    \mlimch(f,g,s,p,\tau^*)\land  \exists\delta\forall\alpha\big(\varphi^\sigma(\alpha)\to|\alpha|\leq_p|\delta|\big).
  \end{displaymath}
  Clearly it holds true if and only if $\tau=\tau^*(p_0)\in I$ and
  $L_{s,p}(f/g)$ is a bounded $\tau^\ZZ$\--chunk. Finally, let
  $\mInt(h,p,\tau^*)$ be the formula
  \begin{displaymath}
    \exists f,g,s,\,\limit(f,g,h,s,p)\land
    \mlimBch(f,g,s,p,\tau^*).
  \end{displaymath}
  If $\CF(X)\models\mInt(h,p,\tau^*)$ then $p$ has a single zero
  $p_0$, $s$ vanishes on a germ at $p_0$ and $h(p_0)\in
  L_{s,p}(f/g)$; further, by the above, $\tau=\tau^*(p_0)\in I$,
  and $L_{s,p}(f/g)$ is a bounded $\tau^\ZZ$\--chunk (in particular it is a bounded discrete subset of $K$). Since $L_{s,p}(f/g)$ is
is a closed definable subset of $K$ it is then finite by \hyperlink{BFin}{(BFin)}.
  Thus by Remark~\ref{re:fin-chunk} (in multiplicative notation) there is a
  non-negative integer $n$ such that $L_{s,p}(f/g)=\{\tau^i\tq -n\leq i\leq n\}$. We obtain
  $L_{s,p}(f/g)\subseteq\tau^\ZZ$ and so $h(p_0)\in\tau^\ZZ$.

  Conversely, assume that $p$ has a single zero $p_0$ and for this zero we have
  $\tau=\tau^*(p_0)\in I$ and  $h(p_0)=\tau^k$ for some
  $k\in\ZZ$. Let $T_h=\{\tau^i\tq -|k|\leq i\leq|k|\}$. Since $X$ has arbitrarily
  many germs
  at $p_0$, we may invoke Property~\ref{py:mdir-global} to obtain functions
  $s_i\in\CF(X)$ that vanish at separated germs at $p_0$, for $-|k|\leq i\leq|k|$\footnote{Here $|k|$
  denotes the ordinary absolute value in $\ZZ$, not the image of $k$ in
  $|K|$.}. Let $s$ be the product of the $s_i$'s and
  $S=\{s=0\}$. By Lemma~\ref{le:limits} there are $f,g\in\CF(X)$ such that
  $g(x)\neq0$ on $S$ and
\[
T_h = L_{s,p}(f/g) = \big\{l(p_0)\tq l\in\CF(X)\mbox{ and }p_0\in\overline{S\cap\{f=gl\}}\big\}.
\]
  Thus $\CF(X)\models\limit(f,g,h,s,p)\land \mlimBch(f,g,s,p,\tau^*)$, hence
  $\CF(X)\models\mInt(h,p,\tau^*)$.
  This finishes the proof of the first equivalence of the theorem.

  \smallskip
  If $K$ has a total order $\preccurlyeq$ that is definable in
  $\Lring\cup\{\cO\}$ and compatible with $+$, then by Proposition~\ref{pr:local-formulas} the following
  relations are definable in $\Lring\cup\{\cO\}$.
  \begin{displaymath}
    g\preccurlyeq_s f \iff \{s=0\}\subseteq\{g\preccurlyeq f\}
  \end{displaymath}
  \begin{displaymath}
    g\prec_s f \iff \{s=0\}\subseteq\{g\prec f\}
  \end{displaymath}
  Let
  $\Phi^\sigma(\alpha,\beta,\gamma,p,\tau^*)$ be the conjunction of the following formulas.
  \begin{enumerate}[1.]
    \item
      $\varphi^\sigma(-\alpha')$.
    \item
      $-\gamma\preccurlyeq_p \alpha+\beta\preccurlyeq_p \gamma \to \varphi^\sigma(\alpha+\beta)$.
    \item
      $\displaystyle \forall u\,
        \big[-\gamma\preccurlyeq_p u\preccurlyeq_p \gamma \to
        \exists\xi\,\big(\varphi^\sigma(\xi)\land \xi\preccurlyeq_p u \prec_p \xi+\tau^*\big)\big]$.
    \item
      $\displaystyle \forall u,\xi,\xi'\,
        \big(B(u,\xi,p,\tau^*)\land B(u,\xi',p,\tau^*)\big]
        \to p\sqsubseteq \xi-\xi'\big)$,\\
        where $B(u,\xi,p,\tau^*)$ stands for $\varphi^\sigma(\xi)\land \xi \preccurlyeq_p u \prec_p \xi+\tau^*$.
  \end{enumerate}
  We then let $\alimch(f,g,s,p,\tau^*)$ be the conjunction
  of $0\prec_p \tau^*$
  with $\varphi^\sigma(\tau^*)$ and
  \begin{displaymath}
    \forall \alpha,\beta,\gamma\,
    \big[\varphi^\sigma(\alpha)\land \varphi^\sigma(\beta)\land \varphi^\sigma(\gamma) \to \Phi^\sigma(\alpha,\beta,\gamma,p,\tau^*)\big],
  \end{displaymath}
  and let $\alimBch(f,g,s,p,\tau^*)$ be
  \begin{displaymath}
    \alimch(f,g,s,p,\tau^*)\land  \exists\delta\forall\alpha\big(\varphi^\sigma(\alpha)\to\alpha \preccurlyeq_p \delta\big).
  \end{displaymath}
  Finally we take $\aInt(h,p,\tau^*)$ as $\exists f,g,s \bigl(\limit(f,g,h,s,p)\land
  \mlimBch(f,g,s,p,\tau^*)\bigr)$.
  The proof that it satisfies the second equivalence in the Theorem
  is analogous  to the multiplicative case and left to the reader.
\end{proof}

\begin{corollary}\label{co:def-Z-p0}
  Assume that $\cK$ satisfies
  \hyperlink{BFin}{(BFin)}
  and take $\tau\in K^\times$ with
  $|\tau|<|1|$. Let $X\subseteq K^m$ be a definable set having arbitrarily many
  germs at a given point $p_0$. Let $p,\tau^*\in\CF(X)$ be such that $p_0$ is
  the unique zero of $p$ and $\tau^*(p_0)=\tau$. Then the ring of integers
  $(\ZZ,+,\times)$ is interpretable in $(\CF(X),\cB)$ with parameters
  $(p,\tau^*)$.
\end{corollary}

\begin{proof}
  For every $k,l\in\ZZ$ let $l|k$ denote the divisibility relation. It is
  well known (see for example \cite{denis-1989})
  that multiplication is 0-definable in
  $(\ZZ,+,|,\leq)$, hence it suffices to interpret the latter structure in
  $(\CF(X),\cB)$. Let $\cZ=\{f\in\CF(X)\tq f(p_0)\in\tau^\ZZ\}$, and for
  every $f\in\cZ$ let $\sigma(f)$ be the unique $k\in\ZZ$ such that
  $f(p_0)=\tau^k$.
  Then
  \begin{displaymath}
    f(p_0)\in\tau^\ZZ\iff (\CF(X),\cB)\models\mInt(f,p,\tau^*),
  \end{displaymath}
  by Theorem~\ref{th:def-chunk} and therefore
$\cZ$ is definable in $(\CF(X),\cB)$.
The equivalence relation $\sigma(f)=\sigma(g)$ on $\cZ$ is also definable, because
  \begin{displaymath}
    \sigma(f)=\sigma(g)\iff f(p_0)=g(p_0) \iff (\CF(X),\cB)\models p\sqsubseteq f-g.
  \end{displaymath}
  For every
  $f,g\in\cZ$ we have obviously $\sigma(f)+\sigma(g)=\sigma(fg)$ and
  \begin{displaymath}
    \sigma(g)\geq \sigma(f) \iff |g(p_0)|\leq|f(p_0)|\iff
    (\CF(X),\cB)\models|g|\leq_p|f|.
  \end{displaymath}
  This gives an interpretation of $(\ZZ,+,\leq)$ in $(\CF(X),\cB)$.
  Moreover $\sigma(g)$ divides $\sigma(f)$ if and only if
  $f(p_0)\in(\tau^{\sigma(g)})^\ZZ=(g(p_0))^\ZZ$.
  We obtain
  \begin{displaymath}
    \sigma(f)|\sigma(g)\iff (\CF(X),\cB)\models\mInt(f,p,g)
  \end{displaymath}
  in the terminology of Theorem~\ref{th:def-chunk}.
  Consequently $(\ZZ,+,|,\leq)$ is interpretable in $(\CF(X),\cB)$, from which
  the result follows.
\end{proof}

\smallskip\noindent
We say that $|K^\times|$ is {\df $\tau$\--archimedean} for some $\tau\in K$ if
$\tau^\ZZ$ is coinitial in $|K^\times|$, that is for every $x\in K^\times$ there is an
integer $k$ such that $|\tau^k|\leq|x|$. In the ordered case, $\cK$ is
$2$\--archimedean if and only if it is archimedean in the sense
that $\ZZ$ is cofinal in $K$.

\begin{corollary}\label{co:plong-non-elem}
  Let $\cK'=(K',\dots)$ be any elementary extension of $\cK$ and assume
  that $\cK$ and $\cK'$ satisfy \hyperlink{BFin}{(BFin)}. Let $X\subseteq K^m$ be a definable set
  having arbitrarily many germs at a given point $p_0$. Let $X'\subseteq K'^m$ be
  defined by a formula defining $X$. If $|K^\times|$ is $\tau$\--archimedean
  for some $\tau\in K$ and $|K'^\times|$ is not $\tau$\--archimedean, then the
  natural $\Lring\cup\{\cB\}$\--embedding of $\CF(X)$ into $\CF(X')$ is \emph{not}
  an elementary embedding.
\end{corollary}

\begin{proof}
  Let $p\in\CF(X)$ be such that $\{p=0\}=\{p_0\}$, and $\tau^*$ be the constant
  function on $X$ with value $\tau$. The assumption that $|K^\times|$ is
  $\tau$\--archimedean implies that $|0|<|\tau|<|1|$. For every $f\in\CF(X)$,
  there exists $k\in\ZZ$ such that $|f(p_0)|\leq|\tau^k|$, so there exists
  $g\in\CF(X)$ such that $|f(p_0)|\leq|g(p_0)|$ and $g(p_0)\in\tau^\ZZ$ (it
  suffices to take $g=(\tau^*)^k$). Thus by Theorem~\ref{th:def-chunk}
\[
    (\CF(X),\cB)\models\forall f\, \exists g,\, |f|\leq_p|g|\land\mInt(g,p,\tau^*).\leqno{(*)}
\]
  On the other hand, since $K'$ is not $\tau$\--archimedean, there exists
  an element $a'\in K'$ such that $|a'|\nleq|\tau^k|$ for every $k\in\ZZ$. Let $f'$
  be the constant function on $X$ with value $a'$. For every
  $g'\in\CF(X')$ such that $g'(p_0)\in\tau^\ZZ$, $|f'(p_0)|\nleq|g'(p_0)|$, hence
  the formula in ($*$) is not satisfied in
  $(\CF(X'),\sqsubseteq,\cB)$.
\end{proof}

\begin{remark}\label{re:CofXNotInLattice}
Let $R\subseteq S$ be real closed fields.
Let $L(R^n)$ be the lattice of closed and semi-algebraic subsets of
$R^n$. In \cite{astier-2013} it is shown that the natural embedding $L(R^n)\to L(S^n)$
is an elementary map.
However,
if $n\geq 2$ and $R=\RR\subsetneq S$, then we know from \ref{co:plong-non-elem}
that the natural embedding $C(R^n)\to C(S^n)$ is not elementary.
Consequently for real closed fields $R$, there is no
interpretation of the ring $C(R^n)$ in the lattice $L(R^n)$ that is independent of $R$.
\end{remark}

\section{Defining integers using local dimension}
\label{se:loc-dim}

In this section, assuming \hyperlink{Dim}{(Dim)}, we show that the set of
functions in $\CF(X)$ that take values in $\tau^\ZZ$ or $\tau\ZZ$ for some
$\tau\in K^\times$ at some point $p_0$ as in Theorem~\ref{th:def-chunk}, is definable in $(\CF(X),\cB)$,
provided $p_0$ is not of small local dimension. In order to
do so we first prove that the points at which $X$ has arbitrarily many
germs are dense among those at which $X$ has local
dimension $\geq2$.

Recall from Definition~\ref{de:defnDim} that $\Delta_k(X)$ denotes the set of $x\in X$ of local dimension $k$
and $W_k(X)$ denotes the set of $x\in X$ such that there is an
open ball $B$ centered at $x$ and a coordinate projection $\pi:K^m\to K^k$
which induces by restriction a homeomorphism between $B\cap X$ and an
open subset of $K^k$.

\begin{lemma}\label{le:Sk-dense}
  Assume that $\cK$ satisfies \hyperlink{Dim}{(Dim)}. For every definable set $X\subseteq K^m$
  and every integer $k\geq0$, $W_k(X)$ is a dense subset of $\Delta_k(X)$. If
  non-empty, both of them have dimension $k$.
\end{lemma}

\begin{proof}
  We already now that $W_k(X)\subseteq\Delta_k(X)$ by Property~\ref{py:Sk-Dk}.
  Pick  $x\in\Delta_k(X)$ and an open ball $B\subseteq K^m$ centered at $x$.
  By shrinking $B$ if necessary we may assume that
  $\dim (B\cap X)=k$. From (Dim\ref{it:dim-Sk}) we know $W_k(B\cap X)\neq\emptyset$.
  On the other hand,
  $W_k(B\cap X)\subseteq W_k(X)$ because $B\cap X$ is open in $X$.
  Consequently $W_k(B\cap X)\subseteq B\cap W_k(X)$ and so $B\cap W_k(X)\neq\emptyset$.
  This proves density.

  By (Dim\ref{it:dim-frontier}) it only remains to prove that
   $W_k(X)$ has dimension $k$, provided it is not empty. Clearly $\dim W_k(X)\geq k$.
   If $\dim W_k(X)=l>k$ then by
  (Dim\ref{it:dim-Sk}) $W_l(W_k(X))$ is non-empty. But $W_k(X)$ is
  open in $X$, hence $W_l(W_k(X))$ is contained in $W_l(X)$. So
  $W_l(W_k(X))$ is contained both in $W_l(X)$ and in $W_k(X)$, a
  contradiction since $W_l(X)$ and $W_k(X)$ are disjoint by Property~\ref{py:Sk-Dk}.
\end{proof}

\medskip\noindent
For any two subsets $A,B$ of a topological space $X$, the {\df strong
order} $\Subset$ is defined by
\begin{displaymath}
  B\Subset A \iff B\subseteq\overline{A\setminus B}.
\end{displaymath}
If $X\subseteq K^m$ is definable and $f,g\in\CF(X)$ we define
\begin{displaymath}
  g\Subset f \iff \{g=0\}\Subset\{f=0\}.
\end{displaymath}

\begin{lemma}\label{le:def-strong}
  For every definable set $X\subseteq K^m$ and every $f,g\in\CF(X)$
  \begin{displaymath}
    g\Subset f \iff \CF(X)\models\forall h\,(f\sqsubseteq gh\to g\sqsubseteq h).
  \end{displaymath}
  In particular $\Subset$ is definable in $\Lring$.
\end{lemma}

\begin{proof}
  For every definable set $U\subseteq X$ that is open in $X$ and each $p_0\in U$
   there is a function $h\in\CF(X)$ with $h(p_0)\neq0$
  and $h(x)=0$ on $X\setminus U$ (using Lemma~\ref{le:dirac} it suffices to take $h=\delta_{B_0^c}$ for any
  open ball $B_0$ with center $p_0$ such that $B\cap X$ is contained in
  $U$), hence $p_0\in\{h\neq0\}\subseteq U$. Hence the sets $\{h=0\}$ with $h\in\CF(X)$ form
  a basis of closed sets of the topology of $X$.
  Since $g\Subset f$ just if every closed definable set containing $\{f=0\}\setminus\{g=0\}$ also
  contains $\{g=0\}$, we see that $g\Subset f$ if and only if for all $h\in\CF(X)$
  \begin{displaymath}
    \{f=0\}\setminus\{g=0\}\subseteq\{h=0\}\Rightarrow\{g=0\}\subseteq\{h=0\}.
  \end{displaymath}
  This is equivalent to $\{f=0\}\subseteq\{g=0\}\cup\{h=0\}\ \Longrightarrow\ \{g=0\}\subseteq\{h=0\}$, in other words, $f\sqsubseteq  gh\Longrightarrow g\sqsubseteq h$.
\end{proof}

\begin{proposition}\label{pr:def-Dk}
  Assume that $\cK$ satisfies \hyperlink{Dim}{(Dim)}. For every
  integer $k\geq0$ and every $p_0\in X$ the following are equivalent.
  \begin{enumerate}
    \item\label{it:LocDimGeqk}
      $\dim(X,p_0)\geq k$.
    \item\label{it:LocDimGeqkElem}
      $\forall v\in\CF(X)$, $v(p_0)\neq0\Rightarrow\exists f_0\Subset f_1\Subset \cdots \Subset f_k\in\CF(X)$ such that
      $f_0\notin\CF(X)^\times$ and $\{f_k=0\}$ is disjoint from $\{v=0\}$.
  \end{enumerate}
Furthermore, there is a parameter-free formula $\chi_k(p)$ in
  $\Lring$ such that $\CF(X)\models\chi_k(p)$ if and only if $p$ has a
  single zero $p_0\in X$ and $\dim(X,p_0)=k$.
\end{proposition}

\begin{proof}
\ref{it:LocDimGeqk}$\Rightarrow $\ref{it:LocDimGeqkElem}.
  Assume that $\dim(X,p_0)=l\geq k$ and take $v\in\CF(X)$ with
  $v(p_0)\neq0$. Then $\{v\neq0\}$ is a neighborhood of $p_0$ in $X$ hence by
  Lemma~\ref{le:Sk-dense} there is a point $q_0\in\{v\neq0\}\cap W_l(X)$. Let
  $V$ be a definable neighborhood of $q_0$ in $X$, contained in
  $\{v\neq0\}\cap W_l(X)$ and homeomorphic to an open subset $W$ of $K^l$  via the restriction of a coordinate projection $\pi:K^m\to K^l$.

  Let $\phi=(\phi_1,\ldots,\phi_l):K^l\lra K^l$ be defined by $\phi(y)=y-q_0$. For
  $1\leq i\leq l$ let
  $g_i=\nu_i(\phi_1,\dots,\varphi_i)$ where $\nu_i\in\CF(K^i)$ is the map defined in
  Lemma~\ref{le:dirac} \ref{it:def-norm}. Each $g_i$ is in $\CF(K^l)$ and $\{g_i=0\}$ is the affine
  subspace of $K^m$ defined by $\varphi_1=\cdots=\varphi_i=0$, hence $g_0\Subset\cdots\Subset g_l$. Let
  $B$ be a closed ball in $K^m$ centered at $q_0$ and contained in
  $V$. Take $\delta_B\in\CF(K^m)$ with $\{\delta_B=0\}=B$ as given
  by Lemma~\ref{le:dirac} \ref{it:dirac-ball}. Finally let
  $f_i$ be the restriction of $\nu_2(g_i\circ\pi,\delta_B)$ to $X$, where $\nu_2$ is
  defined in Lemma~\ref{le:dirac} \ref{it:def-norm}. Clearly $f_i\in\CF(X)$, $f_0(q_0)=0$
  hence $f_0\notin\CF(X)^\times$, $\{f_k=0\}\subseteq B$ is disjoint from $\{v=0\}$, and
  $f_0\Subset \cdots\Subset f_k$ because $\pi$ maps the zero-set of each
  $f_i$ homeomorphically to $\{g_i=0\}\cap\pi(B)$.

\noindent
\ref{it:LocDimGeqkElem}$\Rightarrow $\ref{it:LocDimGeqk}.
Let $B_0$ be
  an open ball in $K^m$ centered at $p_0$. We have to prove that
  $\dim (B_0\cap X)\geq k$. Let $v\in\CF(X)$ be such that $v(x)=0$ on $X\setminus B_0$
  and $v(p_0)\neq0$ (e.g. one can choose $v=\delta_{B_0^c}$ as in
  Lemma~\ref{le:dirac} \ref{it:dirac-co-ball}). By assumption
  there are $f_0\Subset\cdots\Subset f_k\in\CF(X)$ such that $f_0\notin\CF(X)^\times$ and $\{f_k=0\}$
  is disjoint from $\{v=0\}$, hence contained in $B_0\cap X$. It then
  suffices to prove that each set $Z_i=\{f_i=0\}$ has dimension $\geq i$.
  This holds for $Z_0$ because $f_0\notin\CF(X)^\times$ implies $Z_0\neq\emptyset$. When $i>0$,
  the set $Z_i$ is non-empty since it contains $Z_0$; further $Z_{i-1}\subseteq
  \partial(Z_i\setminus Z_{i-1})$ because $Z_{i-1}\Subset Z_i$; consequently $\dim Z_{i-1}<\dim Z_i$
  by (Dim\ref{it:dim-frontier}). Item \ref{it:LocDimGeqk} now follows by induction on
  $i$.

  Using Lemma~\ref{le:def-strong} and Lemma~\ref{le:PtIntIsol},
  the equivalence of \ref{it:LocDimGeqk} and \ref{it:LocDimGeqkElem} implies
  the existence of a parameter-free formula $\chi_{\geq k}(p)$ in
  $\Lring$, such that for every $p\in \CF(X)$ we have
  $\CF(X)\models\chi_{\geq k}(p)$ just
  if $p$ has a single zero $p_0\in X$ and $\dim(X,p_0)\geq k$.
  We then take $\chi_k$ as $\chi_{\geq k}\land \lnot\chi_{\geq k+1}$.
\end{proof}

\begin{proposition}\label{pr:mdir-open}
  Every definable open subset $U$ of $K^r$ with $r\geq2$ has arbitrarily
  many germs at every point $p_0\in U$.
\end{proposition}

\begin{proof}
  For each $k\geq 1$ we need to define $s_i$ and $\delta_i$ verifying
  conditions (S1)-(S4) of Definition~\ref{de:vanish-on-sep-germs}.
  It suffices to do the case of $K^r$ (and then take restrictions to $U$ of the functions
  $s_i$ and $\delta_i$ built for $K^r$).
  Let $V_k=K^r$, take a hyperplane $H\subset K^r$
  containing $p_0$ and $k$ distinct lines $L_1,\dots,L_k$ passing through
  $p_0$ and not contained in $H$. Let $\sigma_i:K^r\to L_i$ be the projection
  onto $L_i$ along $H$. Now let $s_i(x)=\nu_r(x-\sigma_i(x))$ where
  $\nu_r$ is given in Lemma~\ref{le:dirac} \ref{it:def-norm}.
  Conditions (S1) and (S2) are
  fulfilled since $\{s_i=0\}=L_i$ for each $i$.
  Finally let $\delta_i=\prod_{j\neq j}\delta_{i,j}$ where
  $\delta_{i,j}(x)=s_j(x)/s_j(\sigma_i(x))$. Clearly $\delta_{i,j}(x)\in\CF(K^r\setminus\{p_0\})$
  and we have $\delta_{i,j}(x)=1$ on $L_i\setminus\{p_0\}$ and $\delta_{i,j}(x)=0$ on
  $L_j\setminus\{p_0\}$, hence $\delta_i$ has properties (S3) and (S4).
\end{proof}

\medskip\noindent
The intuition coming from real geometry suggests that having local
dimension $\geq2 $ should be necessary (and sufficient) for a definable
set to have arbitrarily many germs at a point. In contrast, the
next result shows that local dimension $\geq1$ (which is obviously
necessary) is sufficient at least if $\cK$ has the following property,
which holds true for example in $P$\--minimal structures like the $p$\--adics.
\begin{description}
  \item[(Z)]\hypertarget{ConditionZ}{}
    $v(K^\times)$ is elementarily equivalent to $(\ZZ,+,\leq)$ (i.e., it is a $\ZZ$\--group)
    and for every definable set $X\subseteq K^\times$, $v(X)$ is definable in
    $v(K^\times)$.
\end{description}

\begin{proposition}\label{pr:mdir-Z}
  If $\cK$ satisfies \hyperlink{ConditionZ}{(Z)}, then every definable open subset $U$ of $K$
  has arbitrarily many germs at every point of $U$.
\end{proposition}

\begin{proof}
  For any $p_0\in U$, the set $\{v(u-p_0)\tq u\in U\setminus\{p_0\}\}$ is a Presburger
  set by \hyperlink{ConditionZ}{(Z)}, and it is not bounded above in $v(K^\times)$. Hence it
  contains some set
  \begin{displaymath}
    A=\big\{\xi\in v(K^\times)\tq \alpha\leq \xi\mbox{ and }\xi\equiv a\,[N]\big\}
  \end{displaymath}
  for some $\alpha\in v(K^\times)$ and some integers $a$ and $N\geq1$ (where $\xi\equiv
  a\,[N]$ denotes the usual congruence relation $\xi-a\in Nv(K^\times)$). Given
  an integer $k\geq1$, for $1\leq i\leq k$ let
  \begin{displaymath}
    A_i=\big\{\xi\in A\tq \alpha\leq \xi\mbox{ and }\xi\equiv a+iN\,[kN]\big\}
  \end{displaymath}
  and let $S_i=v^{-1}(A_i)\cap U$. The sets $S_i$ are pairwise disjoint,
  clopen in $U\setminus\{p_0\}$, and $p_0$ belongs to the closure of each of
  them. So the functions defined by $s_i(x)=0$ if $x\in S_i$ and
  $s_i(x)=\delta_{p_0}(x)$ otherwise (where $\{\delta_{p_0}=0\}=\{p_0\}$, see
  Lemma~\ref{le:dirac}) are continuous on $U$. Clearly each of them
  vanishes on a germ at $p_0$ since $\{s_i=0\}=S_i\cup\{p_0\}$.
\end{proof}

\begin{remark}\label{re:mdir-locdim}
  For all definable sets $X\subseteq K^m$, $Y\subseteq K^n$ and every point $p_0\in
  X$, if there is a definable homeomorphism $\varphi:U\to V$ such that $U$
  (resp. $V$) is a definable neighborhood of $p_0$ in $X$ (resp. of
  $\varphi(p_0)$ in $Y$) then $X$ has arbitrarily many germs at $p_0$ if and
  only if $Y$ has arbitrarily many germs at $\varphi(p_0)$, because this
  property is local. It then follows from Proposition~\ref{pr:mdir-open}
  that for every definable set $X\subseteq K^m$ and every integer $k\geq2$, $X$
  has arbitrarily many germs at every point of $W_k(X)$ for every $k\geq2$ (and
  even for $k=1$ if $\cK$ satisfies \hyperlink{ConditionZ}{(Z)} by Proposition~\ref{pr:mdir-Z}).
\end{remark}

For every integer $k\geq0$, every definable set $X\subseteq K^m$ and every $Y\subseteq K$
let $\CF_k(X,Y)$ (resp. $\CF_{\geq k}(X,Y)$) be the set of functions
$f\in\CF(X)$ such that $f(x)\in Y$ for every $x\in\Delta_k(X)$ (resp. $x\in\bigcup_{l\geq
k}\Delta_l(X)$), and $f(x)=0$ otherwise.

\begin{theorem}\label{th:def-fns-Dk-Z}
  Assume that $\cK$ satisfies \hyperlink{Dim}{(Dim)} and \hyperlink{BFin}{(BFin)}. Let $X\subseteq K^m$ be a
  definable set and let $\tau$ be a non-zero element of $K$.
  \begin{enumerate}
    \item
      If $|\tau|\neq|1|$ then the set $\CF_k(X,\tau^\ZZ)$ is definable in
      $\Lring\cup\{\cB\}$ for every integer $k\geq2$, hence so is $\CF_{\geq
      2}(X,\tau^\ZZ)$. If moreover $\cK$ satisfies \hyperlink{ConditionZ}{(Z)} the same holds true
      for $k\geq1$.
    \item
      If $K$ has a total order $\preccurlyeq$ compatible with $+$ and
      definable in $\Lring\cup\{\cO\}$ then the set
      $\CF_k(X,\tau\ZZ)$ is definable in $\Lring\cup\{\cB\}$ for every
      integer $k\geq2$, hence so is $\CF_{\geq 2}(X,\tau\ZZ)$.
  \end{enumerate}
\end{theorem}

By Remark~\ref{re:def-O3}, Theorem~\ref{th:def-fns-Dk-Z} leads to
definability in $\Lring$ in many cases (in particular if $\cK$ is an
expansion of a local field or if the theory of $\cK$ is dp-minimal)
provided $K$ is not algebraically closed.

\begin{remark}\label{re:def-fns-Dk-Z}
  The formulas given by the proof of Theorem~\ref{th:def-fns-Dk-Z}
  involve only one parameter: the function with constant value $\tau$.
  Thus if $\tau\in\ZZ$ these definitions of $\CF_k(X,\tau^\ZZ)$ and $\CF_k(X,\tau\ZZ)$
  are parameter-free.
\end{remark}

\begin{proof}
  Replacing $\tau$ by $1/\tau$ if necessary we may assume that
  $|\tau|<|1|$.
  Let $\chi_k(p)$ and $\mInt(h,p,\tau^*)$ be the formulas given by
  Proposition~\ref{pr:def-Dk} and Theorem~\ref{th:def-chunk} respectively.
  Consider the formula $\FORM{Z}_k(h,\tau^*)$ defined as
  \begin{displaymath}
    \forall p\, \big[\chi_k(p)\to \forall v\, \big( p\not\sqsubseteq v
      \to \exists q\, [q\not\sqsubseteq v\land \mInt(h,q,\tau^*)]\big)\big].
  \end{displaymath}
  Given $\tau\in K^\times$ with $|0|<|\tau|<|1|$ let $\tau^*\in\CF(X)$ be the
  function with constant value $\tau$. Let $h$ be any function in
  $\CF(X)$ and $D_h=\{x\in X\tq h(x)\in\tau^\ZZ\}$.

  Assume that $\CF(X)\models\FORM{Z}_k(h,\tau^*)$. For any point $p_0\in\Delta_k(X)$
  and any open ball $B_0\subseteq K^m$ centered at $p_0$ let $p=\delta_{p_0}$ and
  $v=\delta_{B_0^c}$ be the functions from Lemma~\ref{le:dirac}. Then $\CF(X)$
  satisfies $\chi_k(p)$ and $p\not\sqsubseteq v$; hence $\FORM{Z}_k(h,\tau^*)$ gives
  $q\in\CF(X)$ having a single zero $q_0$ with $q_0\in B$ and
  $h(q_0)\in\tau^\ZZ$. In other words $p_0$ belongs to the closure of
  $D_h$. By continuity we obtain $h(p_0)\in\tau^\ZZ$, hence
  $h\in\CF_k(X,\tau^\ZZ)$.

  Conversely assume that $h\in\CF_k(X,\tau^\ZZ)$. Take $p\in\CF(X)$ with
  $\CF(X)\models\chi_k(p)$. Then $p$ has a single zero $p_0$ and $p_0\in\Delta_k(X)$.
  For any $v\in\CF(X)$ with $\CF(X)\models p\not\sqsubseteq v$, by continuity, the
  set $V=\{v\neq0\}$ is a neighborhood of $p_0$ in $X$. Hence by
  Lemma~\ref{le:Sk-dense} there is a point $q_0\in V\cap W_k(X)$. In
  particular $q_0\in\Delta_k(X)$ by Property~\ref{py:Sk-Dk} and therefore $h(q_0)\in\tau^\ZZ$.
  By Remark~\ref{re:mdir-locdim} the assumptions on $\cK$ ensure that
  $X$ has arbitrarily many germs at $p_0$. Thus $\CF(X)\models\mInt(h,q,\tau^*)$
  with $q=\delta_{q_0}$ defined in
  Lemma~\ref{le:dirac} \ref{it:dirac-pt}, which shows that $\CF(X)\models\FORM{Z}_k(h,\tau^*)$.

  The proof of the second statement is similar. Replacing $\tau$ by
  $-\tau$ if necessary we may assume that $\tau\succ0$. Then let $\tau^*\in\CF(X)$ be
  the function with constant value $\tau$ and replace $\mInt(h,q,\tau^*)$
  by $\aInt(h,q,\tau^*)$ in $\FORM{Z}_k(h,\tau^*)$.
\end{proof}

\def\DIRroot{}% Marcus
\def\href#1#2{}% Marcus

%  \bibliographystyle{alpha}
%  \bibliography{biblio2}

\end{document}